\documentclass[11pt]{article}

\usepackage{epsfig,epsf,fancybox}
\usepackage{graphicx,graphics}
\usepackage{epsf,epstopdf}
\usepackage{color}
\usepackage{graphicx,psfrag,amsmath,amssymb}
\usepackage{algorithm}
\usepackage{amsmath}
\usepackage{MnSymbol}
\usepackage{xy}
\usepackage{algorithmic}
\usepackage{enumerate}
\usepackage{boxedminipage}
\usepackage{graphicx,xcolor}
\usepackage{amsmath,amsxtra,amsfonts,amscd,amssymb,bm}
\usepackage{multirow,cite}
\usepackage[algo2e,linesnumbered,vlined,ruled]{algorithm2e}

\numberwithin{table}{section}
\newtheorem{theorem}{Theorem}[section]
\newtheorem{corollary}[theorem]{Corollary}
\newtheorem{lemma}[theorem]{Lemma}
\newtheorem{proposition}[theorem]{Proposition}
\newtheorem{definition}[theorem]{Definition}
\newtheorem{example}{Example}[section]

\newtheorem{remark}[theorem]{Remark}

\newcommand{\be}{\begin{equation}}
\newcommand{\ee}{\end{equation}}
\newcommand{\bee}{\begin{equation*}}
\newcommand{\eee}{\end{equation*}}
\newcommand{\bea}{\begin{eqnarray}}
\newcommand{\eea}{\end{eqnarray}}
\newcommand{\beaa}{\begin{eqnarray*}}
\newcommand{\eeaa}{\end{eqnarray*}}

\newcommand{\R}{\mathbb{R}}

\newcommand{\cA}{\mathcal{A}}

\newcommand{\cX}{\mathcal{X}}
\newcommand{\cB}{\mathcal{B}}

\newcommand{\cF}{\mathcal{F}}

\newcommand{\cH}{\mathcal{H}}

\newcommand{\cO}{{\mathcal{O}}}

\newcommand{\prox}{\mathbf{prox}}

\newcommand{\argmin}{\mathop{\rm argmin}}
\newcommand{\argmax}{\mathop{\rm argmax}}
\newcommand{\half}{\frac{1}{2}}

\begin{document}

\title{On Lower Iteration Complexity Bounds for the \\
	Convex Concave Saddle Point Problems}

\author{
Junyu Zhang\thanks{Department of Electrical and Computer Engineering, Princeton University, junyuz@princeton.edu  $~\qquad\qquad\qquad~\,\,\,$ 
Department of Industrial Systems Engineering and Management, National University of Singapore}
\and
Mingyi Hong\thanks{Department of Electrical and Computer Engineering, University of Minnesota, mhong@umn.edu}
\and
Shuzhong Zhang\thanks{Department of Industrial and Systems Engineering, University of Minnesota, zhangs@umn.edu}
}
\date{\today}
\maketitle
\begin{abstract}
In this paper, we study the lower iteration complexity bounds for finding the saddle point of 
a strongly convex and strongly concave saddle point problem: $\min_x\max_yF(x,y)$. We restrict the classes of algorithms in our investigation to be either pure first-order methods or methods using proximal mappings. For problems with gradient Lipschitz constants ($L_x, L_y$ and $L_{xy}$) and strong convexity/concavity constants ($\mu_x$ and $\mu_y$),
the class of pure first-order algorithms with the linear span assumption is shown to have a  lower iteration complexity bound of $\Omega\,${\footnotesize$\left(\sqrt{\frac{L_x}{\mu_x}+\frac{L_{xy}^2}{\mu_x\mu_y}+\frac{L_y}{\mu_y}}\cdot\ln\left(\frac{1}{\epsilon}\right)\right)$},   where the term {\small$\frac{L_{xy}^2}{\mu_x\mu_y}$} explains how the coupling influences the iteration complexity. Under several special parameter regimes, this lower bound has been achieved by corresponding optimal algorithms. However, whether or not the bound under the general parameter regime is optimal remains open. Additionally, for the special case of bilinear coupling problems, given the availability of certain proximal operators, a lower bound of {\footnotesize$\Omega\left(\sqrt{\frac{L_{xy}^2}{\mu_x\mu_y}}\cdot\ln(\frac{1}{\epsilon})\right)$} is established under the linear span assumption, and optimal algorithms have already been developed in the literature. By exploiting the orthogonal invariance technique, we extend both lower bounds to the general pure first-order algorithm class and the proximal algorithm class without the linear span assumption. As an application, we apply proper scaling to the worst-case instances, and we derive the lower bounds for the general convex-concave problems with $\mu_x = \mu_y = 0$. Several existing results in this case can be deduced from our results as special cases. 

\vspace{0.3cm}

\end{abstract}

\section{Introduction}
In this paper, we establish a lower iteration complexity bound for the first-order methods to solve the following min-max saddle point problem
\begin{equation}
\label{prob:min-max}
\min_x \max_y  F(x,y),
\end{equation}
%Such a problem has many important applications
which is of fundamental importance in, e.g.,
game theory \cite{von2007theory,nisan2007algorithmic}, image deconvolution problems \cite{chambolle2011first}, parallel computing \cite{xiao2019dscovr}, adversarial training \cite{goodfellow2014generative,arjovsky2017wasserstein}, and statistical learning \cite{abadeh2015distributionally}.

To proceed, let us introduce the following two problem classes.

\begin{definition}[Problem class $\cF(L_{x},L_{y},L_{xy},\mu_x,\mu_y)$]
	\label{defn:general-class}
	$F(\cdot,y)$ is  $\mu_x$-strongly convex for any fixed $y$ and $F(x,\cdot)$ is  $\mu_y$-strongly concave for any fixed $x$. Overall, the function $F$ is smooth and $\nabla F$ satisfies the following Lipschitz continuity condition
	\begin{eqnarray}
	\label{defn:Lip}
	\begin{cases}
	\|\nabla_x  F(x_1,y) - \nabla_x F(x_2,y)\|\leq L_x\|x_1-x_2\|,&\forall x_1,x_2,y     \\
	\|\nabla_y  F(x,y_1) - \nabla_y F(x,y_2)\|\leq L_y\|y_1-y_2\|,&\forall x,y_1,y_2     \\
	\|\nabla_x  F(x,y_1) - \nabla_x F(x,y_2)\|\leq L_{xy}\|y_1-y_2\|,&\forall x,y_1,y_2  \\
	\|\nabla_y  F(x_1,y) - \nabla_y F(x_2,y)\|\leq L_{xy}\|x_1-x_2\|,&\forall x_1,x_2,y. \\
	\end{cases}
	\end{eqnarray}
\end{definition}
We shall remark here that the constants in \eqref{defn:Lip} may also be understood as the bounds on the different blocks of the Hessian matrix $\nabla^2 F(x,y)$ if $F$ is twice continuously differentiable. That is,
\begin{equation*}
\sup_{x,y}\|\nabla_{xx}^2  F(x,y)\|_2\leq L_x,\quad \sup_{x,y}\|\nabla_{yy}^2  F(x,y)\|_2\leq L_y,\quad  \sup_{x,y}\|\nabla_{xy}^2  F(x,y)\|_2\leq L_{xy}.
\end{equation*}
However, throughout this paper we do not assume either $F(\cdot, y)$ or $F(x,\cdot)$ is second-order differentiable.

The second problem class is the bilinear saddle point model:
\begin{definition}[Bilinear  class $\cB(L_{xy},\mu_x,\mu_y)$]
	\label{defn:bilinear-class}
	In this special class, the problems are written as
	\begin{equation}
	\label{prob:bilinear}
	\min_x \max_y  F(x,y):=f(x) + x^\top Ay - g(y),
	\end{equation}
%	We require
where $f(x)$ and $g(y)$ are both lower semi-continuous with $f(x)$ being $\mu_x$-strongly convex and $g(y)$ being $\mu_y$-strongly convex. The coupling matrix $A$ satisfies $\|A\|_2\leq L_{xy}$.
\end{definition}
For this special model class $\cB(L_{xy},\mu_x,\mu_y)$, %subclass of problems,
we assume the availability of the following prox-operations:
\begin{equation}
\label{defn:prox-fg}
\prox_{\gamma f}(v):=\argmin_x f(x) + \frac{1}{2\gamma}\|x-v\|^2 \quad \mbox{ and } \quad\prox_{\sigma g}(u):=\argmin_y g(y) + \frac{1}{2\sigma}\|y-u\|^2.
\end{equation}

%For the problem classes $\cF(L_{x},L_{y},L_{xy},\mu_x,\mu_y)$  and $\cB(L_{xy},\mu_x,\mu_y)$ (with the proximal oracles \eqref{defn:prox-fg}),
In this paper we shall establish the lower iteration complexity bound
\[
\Omega\left(\sqrt{\frac{L_x}{\mu_x}+\frac{L_{xy}^2}{\mu_x\mu_y}+\frac{L_y}{\mu_y}}\cdot\ln\left(\frac{1}{\epsilon}\right)\right) \mbox{ for } \cF(L_{x},L_{y},L_{xy},\mu_x,\mu_y),
\]
and
\[
\Omega\left(\sqrt{\frac{L_{xy}^2}{\mu_x\mu_y}}\cdot\ln\left(\frac{1}{\epsilon}\right)\right) \mbox{ for } \cB(L_{xy},\mu_x,\mu_y)
\]
with the proximal oracles \eqref{defn:prox-fg}. In particular, we first establish these lower bounds for pure first-order and general proximal algorithm classes under the linear span assumption. Later on we generalize the results for more general algorithm classes without the linear span assumption through the orthogonal invariance technique introduced by \cite{nemirovsky1992information}. For more detailed applications of the orthogonal invariance technique in the lower bound derivation, the interested readers are referred to \cite{carmon2017lower-1,carmon2017lower-2,YYXu}.

As an application of the above bound, we apply proper scaling to the worst-case instances and show that the above result implies several exisiting lower bounds for general convex-concave problems with bounded saddle point solutions. Specifically, we have 
\[
\Omega\left(\sqrt{\frac{L_xR_x^2}{\epsilon}}+\frac{L_{xy}R_xR_y}{\epsilon}+\sqrt{\frac{L_yR_y^2}{\epsilon}}\,\,\right) \mbox{ for } \cF(L_{x},L_{y},L_{xy},0,0), \,\,\mbox{and}\,\,\|x^*\|\leq R_x, \|y^*\|\leq R_y,
\]
and
\[
\Omega\left(\frac{L_{xy}R_xR_y}{\epsilon}\right) \mbox{ for } \cB(L_{xy},0,0),\,\,\mbox{and}\,\, \|x^*\|\leq R_x, \|y^*\|\leq R_y.
\]
For the above two lower bounds, we remark that under specific parameter regimes, the first bound is known, see \cite{nemirovsky1992information} for the case with $L_x = L_y = L_{xy}$ and see \cite{YYXu} for the case with $L_y = 0$. However to our best knowledge, the first bound under a general set of parameters as well as the second bound for bilinear problem class are not known. Similar reductions can also be done for the problem classes with only one of $\mu_x$ and $\mu_y$ equal to 0, for which the lower bounds have already been discovered in \cite{YYXu}.

%Such lower bound results provide important theoretical understanding of the perforamance of the min-max algotihms.
Such lower iteration complexity results shed light on understanding the performance of the algorithms designed for min-max saddle point models. There are numerous results in the literature prior to ours.
As a special case of \eqref{prob:min-max}, the lower bound results of convex minimization problem with $F(x,y) = f(x)$ has been well-studied in the past decades. For convex problems, Nesterov's accelerated gradient method have achieved iteration complexities of $\cO(\sqrt{L/\epsilon})$ for $L$-smooth convex problems, and $\cO\left(\sqrt{\frac{L}{\mu}}\cdot\ln\left(\frac{1}{\epsilon}\right)\right)$ for $L$-smooth and $\mu$-strongly convex problems respectively, and both of them are shown to match the lower complexity bound for the first-order methods; see \cite{nesterov2018lectures}.

However, for the min-max saddle-point models, the situation is more subtle. Due to the convex-concave nature of $F$, the vector field
$$G(x,y) = \begin{pmatrix}
\nabla_x F(x,y)\\
-\nabla_y F(x,y)
\end{pmatrix}$$
is monotone. Hence the convex-concave saddle point problem is often studied as a subclass of the variational inequality problems (VIP); see e.g.  \cite{nesterov2007dual,nesterov2006solving,juditsky2011solving,nemirovski2004prox,marcotte1987note,taji1993globally} and references therein. Although there have been plenty of studies on the variational inequalities model, the roles played by different Lischitz constants on the different blocks of variables have not been fully explored in the literature.
%people rarely distinguish the differences of these constants.
%Namely, people set
In other words, often one would denote
$L$ to be an overall Lipschitz constant of the vector field $G$, which is of the order $\Theta(\max\{L_x,L_y,L_{xy}\})$ in our case, and set $\mu$ to be the strong monotonicity parameter of $G$, which is of the order $\Theta(\min\{\mu_x,\mu_y\})$ in our case, and no further distinctions among the parameters would be made. Hence the considered problems are of special instances in $\cF(L,L,L,\mu,\mu)$. Under such settings, many algorithms including the mirror-prox algorithm \cite{nemirovski2004prox}, the extra-gradient methods \cite{korpelevich1976extragradient,mokhtari2019unified}, and the accelerated dual extrapolation\footnote{In Nesterov's original paper \cite{nesterov2006solving}, the author did not give a name to his algorithm.
%According to our discussion in the paper, we give it the current name for convenience of reference.}
For convenience of referencing, in this paper we shall call it {\it accelerated dual extrapolation}.}
\cite{nesterov2006solving} and so on, have all achieved the iteration complexity of $\cO\left(\frac{L}{\mu}\cdot\ln\left(\frac{1}{\epsilon}\right)\right)$, and this complexity is shown to be optimal for first-order methods in solving the problem class $\cF(L,L,L,\mu,\mu)$; see \cite{nemirovsky1983problem}. However, under the more general parameter regime of $\cF(L_x,L_y,L_{xy},\mu_x,\mu_y)$, these methods are not optimal. For example, Nesterov's accelerated dual extrapolation method \cite{nesterov2006solving} has a complexity of $\cO\left(\frac{\max\{L_x,L_{xy},L_y\}}{\min\{\mu_x,\mu_y\}}\cdot\ln\left(\frac{1}{\epsilon}\right)\right)$, even if the algorithm are modified carefully one can only guarantee a complexity of {\small$\cO\left(\sqrt{\frac{L_x^2}{\mu_x^2} + \frac{L_{xy}^2}{\mu_x\mu_y} + \frac{L_y^2}{\mu_y^2}}\cdot\ln\left(\frac{1}{\epsilon}\right)\right)$}, both of which do not match the lower bound provided in this paper. More recently, tighter upper bounds have been derived. In \cite{lin2020near}, the authors consider the problems class of $\cF(L,L,L,\mu_x,\mu_y)$ and achieve an upper bound of {\small$\cO\left(\sqrt{\frac{L^2}{\mu_x\mu_y}}\cdot\ln^3\left(\frac{1}{\epsilon}\right)\right)$}, which matches our lower bound when $L_{x} = L_{y} = L_{xy}=L$ up to a logarithmic term. In \cite{wang2020improved}, the authors consider the general problems class of $\cF(L_x,L_y,L_{xy},\mu_x,\mu_y)$, the proposed algorithm achieves an upper bound of {\small$\cO\left(\sqrt{\frac{L_x}{\mu_x} + \frac{L\cdot L_{xy}}{\mu_x\mu_y} + \frac{L_y}{\mu_y}}\cdot\ln^3\left(\frac{L^2}{\mu_x\mu_y}\right)\cdot\ln\left(\frac{1}{\epsilon}\right)\right)$} where $L = \max\{L_x,L_y,L_{xy}\}$, which almost matches our lower bound for the general problem class.
Despite
%the suboptimality of the existing algorithms
the gap for the general problem class $\cF(L_x,L_{xy},L_y,\mu_x,\mu_y)$, given the availability of proximal operators the authors of \cite{chambolle2011first,chambolle2016ergodic} have derived an algorithm for problem class $\cB(L_{xy},\mu_x,\mu_y)$ with complexity {\small$\cO\left(\sqrt{\frac{L_{xy}^2}{\mu_x\mu_y}}\cdot\ln\left(\frac{1}{\epsilon}\right)\right)$}. We will prove in this paper that this result has matched the theoretical lower complexity bound for its problem and algorithm classes, hence optimal.

For the bilinear problem \eqref{prob:bilinear}, when $f$ is smooth and convex, $g(y) = b^\top y$ is linear, the problem is equivalent to the following convex optimization problem $$\min_x  \{f(x): A^\top x-b = 0\}.$$
Without using projection onto the hyperplane $\{x:A^\top x=b\}$ which requires a matrix inversion, pure first-order methods achieve $\cO(1/\epsilon)$ complexity despite the strong convexity of $f$; see e.g. \cite{gao2017first,xu2017accelerated,ouyang2015accelerated}. Those iteration complexity bounds are shown to match the lower bound provided in \cite{YYXu}. For more details on the lower and upper bounds on this formulation, the interested readers are referred to \cite{YYXu}.  Finally, for the bilinear coupling problem \eqref{prob:bilinear}, the authors of \cite{LowB-Quadratic}  show that a lower bound of {\footnotesize $\cO\left(\sqrt{\frac{L_{xy}^2 + \mu_x\mu_y}{\mu_{xy}^2 + \mu_x\mu_y}}\cdot\ln\left(\frac{1}{\epsilon}\right)\right)$} can be derived, where $\mu_{xy}$ stands for the minimum singular value of the coupling matrix $A$. It is interesting that this result covers the linear convergence phenomenon for pure bilinear saddle point problem \cite{Bilear-LinearCVG} where $f(x)\equiv g(y)\equiv 0$. Another remark is that, due to the special construction of the worst-case instance and algorithm class, \cite{LowB-Quadratic} cannot characterize the impact of $L_x$ and $L_y$ as well as the lower bound for proximal algorithm class.

%\vspace{0.4cm}
%\textbf{Additional literatures.} Here we include some related
%but not closely related literatures.

%In this paper, we are mainly focusing on the lower bound results. This is an important part of the theoretical understanding of the optimization or saddle point algorithms.
Other than studies on the first-order algorithms, there are also %For the convex minimization problems, apart from the aforementioned lower bounds on first-order methods
%\cite{nesterov2018lectures},
studies on the higher-order methods as well. For example, in \cite{arjevani2019oracle}  lower iteration complexity bounds for second-order methods are considered, and in \cite{agarwal2017lower,nesterov2018implementable} lower iteration complexity bounds are presented for general higher-order (tensor) methods. For smooth nonconvex optimization, in \cite{carmon2017lower-2} the iteration complexity lower bounds for first-order methods are considered, while in \cite{carmon2017lower-1} 
%the authors established the lower bounds
that for higher-order methods are considered. 

Another line of research is for the non-conex/concave min-max saddle point problems;
%. For such problems, even the definition of a saddle point is not well understood in  general,
see \cite{jin2019local,jin2019minmax} and the references therein. To guarantee convergence, additional structures are often needed. For example, if one assumes that the solutions of the problem satisfy the Minty variational inequality \cite{lin2018solving} then convergent algorithm can be constructed. Another important situation is when $F$ is concave in $y$. In that case, convergence and iteration complexity to a stationary solution is possible; see e.g. \cite{lu2019hybrid}. For more literatures in this type of problems, we refer the interested readers to \cite{lin2019gradient} and the references therein.

\vspace{0.4cm}

\textbf{Organization.} This paper is organized as follows. In Section \ref{sec:Preliminary}, we introduce two different %define the considered
algorithm classes (with or without proximal-operators).  In Section \ref{sec:lower-bound-proximal}, we construct a worst-case example for problem class $\cB(L_{xy},\mu_x,\mu_y)$ and derive the corresponding lower iteration complexity bound for the algorithm class allowing proximal-operators. An optimal algorithm is discussed in this case. In Section \ref{sec:lower-bound-pure}, we construct the worst-case example for problem class $\cF(L_x,L_y,L_{xy},\mu_x,\mu_y)$ and establish the corresponding lower complexity bound for the first-order method (without any proximal oracles). Optimal algorithms under several special parameter regimes are discussed. Finally, we conclude the paper in Section \ref{sec:conclusion}.

\section{The first-order algorithm classes}\label{sec:Preliminary}
In this section, we discuss some preliminaries for the strongly convex and strongly concave saddle point problem. 
Then, we shall introduce two algorithm classes to set the ground for our discussion, and we shall also note specific known algorithms as representative members in those algorithm classes.
%that we will focus on and a few corresponding example algorithms.

\subsection{Primal function, dual function, and the duality gap}
First, we define $\Phi(\cdot)$ to be the primal function and $\Psi(\cdot)$ to be the dual function of the saddle point problem $\min_x \max_y F(x,y)$, respectively, with the following definitions
\begin{equation}
\label{defn:primal-func}
\Phi(x):=\max_y F(x,y)\qquad\mbox{ and }\qquad\Psi(y):=\min_x F(x,y).
\end{equation}
As the maximum of a class of $\mu_x$-strongly convex function, we know $\Phi(x)$ is a $\mu_x$-strongly convex function. Similarly, $\Psi(y)$ is a $\mu_y$-strongly concave function. We define the duality gap as
$$\Delta(x,y) := \max_{y'} F(x,y')-\min_{x'} F(x',y) = \Phi(x)-\Psi(y).$$
Suppose the unique solution of this min-max problem is $(x^*,y^*)$. By the strong duality theorem, we know for any $x$ and $y$ it holds that 
$$\Phi(x)\geq\min_{x'}\Phi(x') = \Phi(x^*) =F(x^*,y^*) = \Psi(y^*) = \max_{y'}\Psi(y')\geq \Psi(y).$$
Together with the $\mu_x$-strong convexity of $\Phi$ and the $\mu_y$-strong concavity of $\Psi$, we further have 
\begin{equation}
\label{defn:primal-gap-LB}
\Delta(x,y)= \Phi(x)-\Phi(x^*) + \Psi(y^*) - \Psi(y)\geq\frac{\mu_x}{2}\|x-x^*\|^2 + \frac{\mu_y}{2}\|y-y^*\|^2.
\end{equation}

Now, suppose that $(\tilde x_k,\tilde y_k)$ is the approximate solution generated after $k$ iterations of an algorithm. Our aim %main strategy in the later discussion 
is to lower bound the distance between $(\tilde x_k,\tilde y_k)$ and $(x^*,y^*)$. By \eqref{defn:primal-gap-LB}, this would construct a lower iteration complexity bound in terms of the duality gap as well.

\subsection{Proximal algorithm class}
First, let us consider the bilinearly coupled problem class \eqref{prob:bilinear} as introduced in Definition \ref{defn:bilinear-class}: 
\begin{equation*}
\min_x \max_y  F(x,y):=f(x) + x^\top Ay - g(y).
\end{equation*}
For this special problem class, let us consider the lower iteration bound of the algorithm class where the proximal oracles \eqref{defn:prox-fg} are available. 
%(the complexity bound is with respect to the number of times such oracles are called).

\begin{definition}[Proximal algorithm class]
\label{defn:Prox-AlgoClass}
In each iteration, the iterate sequence $\{(x^{k},y^{k})\}_{k=0,1,...}$ are generated so that  $(x^{k},y^{k})\in\cH_x^{k} \times \cH_y^{k}$. These subspaces are generated with $\cH_x^{0} = \mathrm{Span}\{x^0\}, \cH_y^{0} = \mathrm{Span}\{y^0\}$ and
\begin{equation}
\label{eqn:Prox-AlgoClass}\begin{cases}
\cH_x^{k+1} := \mathrm{Span}\{x^i,\prox_{\gamma_i f}(\hat{x}^i-\gamma_iA\tilde{y}^i)~\,: \forall \hat{x}^i\in\cH_x^{i},~ \tilde{y}^i\in\cH_y^i,~ 0\leq i\leq k\}\\
\cH_y^{k+1} := \mathrm{Span}\{y^i,\prox_{\sigma_i g}(\hat{y}^i+\sigma_iA^\top\tilde{x}^i): \forall \tilde{x}^i\in\cH_x^{i},~ \hat{y}^i\in\cH_y^i,~ 0\leq i\leq k\}.
\end{cases}
\end{equation}
\end{definition}
Remark that when applying the proximal oracles, it is not necessary to use the most recent iterate $x^k$ as the proximal center. Neither is it necessary to use the gradients of the coupling term (namely the $A^\top x$ and $Ay$ terms) at the current iterate. Instead, the algorithm class allows the usage of the combination of {\it any}\/ points in the historical search space. We shall also remark that the algorithm class in Definition~\ref{defn:Prox-AlgoClass} does not necessarily need to update $x$ and $y$ at the same time, because setting $x^{k+1}=x^k$ or $y^{k+1} = y^k$ also satisfies Definition \ref{defn:Prox-AlgoClass}. Thus this algorithm class also includes the methods that alternatingly update $x$ and $y$. Below is a sample algorithm in this class.

\begin{example}[Algorithm 3 in \cite{chambolle2011first}]
	\label{algo:ChambolleNPock}
	Initialize with $\gamma = \frac{1}{L_{xy}}\sqrt{\frac{\mu_y}{\mu_x}}$, $\sigma = \frac{1}{L_{xy}}\sqrt{\frac{\mu_x}{\mu_y}}$, and $\theta = \frac{L_{xy}}{2\sqrt{\mu_x\mu_y} + L_{xy}}.$ Set $\tilde x^0 = x^0$. Then the algorithm proceeds as
	\begin{equation}
	\begin{cases}
	y^{k+1} = \prox_{\sigma g}(y^k + \sigma A^\top \tilde{x}^k)\\
	x^{k+1} = \prox_{\gamma f}(x^k - \gamma Ay^{k+1})\\
	\tilde{x}^{k+1} = x^{k+1} + \theta(x^{k+1}-x^k).
	\end{cases}
	\end{equation}
\end{example}
It can be observed that this algorithm takes the alternating order of update, by slightly manipulating the index, it can be written in the form of \eqref{eqn:Prox-AlgoClass} in Definition \ref{defn:Prox-AlgoClass}. The complexity of this method is {\footnotesize$\cO\left(\sqrt{\frac{L_{xy}^2}{\mu_x\mu_y}}\cdot\ln\left(\frac{1}{\epsilon}\right)\right)$}.

\subsection{Pure first-order algorithm class}
In constrast to the previous section, here we consider the more general problem class $\cF(L_x,L_y,L_{xy},\mu_x,\mu_y)$:
$$\min_x\max_y F(x,y).$$
For such problems, we refer to the algorithm class as the {\it pure first-order methods}, meaning that there is no proximal oracle in the design of algorithms in this class.

\begin{definition}[Pure first-order algorithm class]
	\label{defn:Pure-AlgoClass}
	In each iteration, the sequence $\{(x_{k},y_{k})\}_{k=0,1,...}$ is generated so that  $(x^{k},y^{k})\in\cH_x^{k} \times \cH_y^{k}$,  with $\cH_x^{0} = \mathrm{Span}\{x^0\}$, $\cH_y^{0} = \mathrm{Span}\{y^0\}$, and
	\begin{equation}
	\label{eqn:Pure-AlgoClass}
	\begin{cases}
	\cH_x^{k+1} := \mathrm{Span}\{x^i,\nabla_xF(\tilde{x}^i,\tilde{y}^i): \forall \tilde{x}^i\in\cH_x^{i}, \tilde{y}^i\in\cH_y^i, 0\leq i\leq k\}\\
	\cH_y^{k+1} := \mathrm{Span}\{y^i,\nabla_yF(\tilde{x}^i,\tilde{y}^i): \forall \tilde{x}^i\in\cH_x^{i}, \tilde{y}^i\in\cH_y^i, 0\leq i\leq k\}.
	\end{cases}
	\end{equation}
\end{definition}

Similar to our earlier comments on the proximal algorithm class, in this class of algorithms the gradients at any combination of points in the historical search space are allowed. The algorithm class also includes the methods that alternatingly update between $x$ and $y$, or even the double loop algorithms that optimize one side until certain accuracy is achieved before switching to the other side. At that level of generality, it indeed accommodates many updating schemes. To illustrate this point, let us present below some sample algorithms in this class.

The first example is a double loop scheme, in which the primal function $\Phi(x)$ is optimized approximately. Specifically, let $y^*(x) = \argmax_y F(x,y)$, by Danskin's theorem, $\nabla \Phi(x) = \nabla_x F(x,y^*(x))$; see e.g. \cite{bertsekas1997nonlinear,rockafellar1970convex}. Therefore, one can apply Nesterov's accelerated gradient method to minimize $\Phi(x)$. The double loop scheme performs this procedure approximately.
\begin{example}[Double loop schemes, \cite{sanjabi2018solving}]
	\label{algo:double-loop}
	Denote  $\alpha_1 = \sqrt{\frac{\mu_x}{L_{\Phi,x}}}$ and $\alpha_2 = \sqrt{\frac{\mu_y}{L_y}}$, where $L_{\Phi,x} = L_x + \frac{L_{xy}^2}{\mu_y}$ is the Lipschitz constant of $\nabla \Phi(x)$ (see \cite{sanjabi2018solving}). Given $(x^0, y^0)$ and define $\bar{x}^0 = x^0$, the double loop scheme works as follows:
	$$\begin{cases}
	x^{k+1} = \bar{x}^k - \frac{1}{L_{\Phi,x}}\nabla_x F(\bar{x}^k, y^k)\\
	\bar x^{k+1} = x^{k+1} + \frac{1-\sqrt{\alpha_2}}{1+\sqrt{\alpha_2}}(x^{k+1}-x^k)
	\end{cases}\quad\mbox{for}\quad k = 0,1,...,T_1,$$
	where the point $y^k$ is generated by an inner loop of accelerated gradient iterations
	$$\begin{cases}
	w^{t+1} = \bar{w}^t + \frac{1}{L_{y}}\nabla_y F(\bar{x}^k, \bar w^t)\\
	\bar w^{t+1} = w^{t+1} + \frac{1-\sqrt{\alpha_1}}{1+\sqrt{\alpha_1}}(w^{t+1}-w^t)
	\end{cases}\,\,\mbox{for}\quad\,\, t = 0,1,...,T_2\,\,\,\, \mbox{and}\,\,\,\, w^0 = \bar w^0 = y^{k-1}.$$
	 Then, set $y^{k}:=w^{T_2+1}$ to be the last iterate of the inner loop.
\end{example}
For simplicity, we have applied a specific scheme of acceleration \cite{nesterov2018lectures} which does not work for nonstrongly-convex problems. In principle, the FISTA scheme can also be used.  For this scheme, with properly chosen $T_1$ and $T_2$, the iteration complexity $\cO${\footnotesize$\left(\sqrt{\frac{L_x}{\mu_x} + \frac{L_{xy}^2}{\mu_x\mu_y}}\cdot\sqrt{\frac{L_y}{\mu_y}}\ln^2\left(\frac{1}{\epsilon}\right)\right)$} is achievable.

In the following, we also list examples of several single loop algorithms, including the gradient descent-ascent method (GDA), the extra-gradient (EG) method \cite{korpelevich1976extragradient} (a special case of mirror-prox algorithm \cite{nemirovski2004prox}), and the accelerated dual extrapolation (ADE) \cite{nesterov2006solving}.
\begin{example}[Single loop algorithms]
	\label{algo:single-loop}
	Let $L = \max\{L_x,L_y,L_{xy}\}, \mu = \min\{\mu_x,\mu_y\}$. Given the initial solution $(x^0,y^0)$, the algorithms proceed as follows:
	$$(\emph{GDA})\qquad\qquad\qquad\qquad\qquad\begin{cases}
	x^{k+1} = x^k - \eta_1 \nabla_x F(x^k,y^k)\\
	y^{k+1}\, = y^k + \eta_1 \nabla_yF(x^k,y^k)
	\end{cases}\qquad\qquad\qquad\qquad\qquad\qquad\qquad\qquad$$
	$$(\emph{EG})\qquad\begin{cases}
	\tilde{x}^{k+1} = x^k - \eta_2 \nabla_x F(x^k,y^k)\\
	\tilde{y}^{k+1}\, = y^k + \eta_2 \nabla_yF(x^k,y^k)
	\end{cases}\mbox{ and}\quad\,\,
	\begin{cases}
	x^{k+1} = x^k - \eta_2 \nabla_x F(\tilde{x}^{k+1},\tilde{y}^{k+1})\\
	y^{k+1}\, = y^k + \eta_2 \nabla_yF(\tilde{x}^{k+1},\tilde{y}^{k+1})
	\end{cases}\qquad\qquad\qquad\qquad\quad$$
	$$(\emph{ADE})\qquad\qquad\begin{cases}
	x^{k+1} = x^k - \eta_3 \left(\frac{\mu}{L+\mu}\nabla_x F(x^k,y^k)+\frac{L}{L+\mu}\nabla_x F(\tilde{x}^{k+1},\tilde{y}^{k+1})\right)\\
	y^{k+1} \,= y^k + \eta_3 \left(\frac{\mu}{L+\mu}\nabla_yF(x^k,y^k)\,+\frac{L}{L+\mu}\nabla_y F(\tilde{x}^{k+1},\tilde{y}^{k+1})\right)
	\end{cases}\qquad\qquad\qquad\qquad\qquad$$
	where $\eta_1 = \cO\left(\frac{\mu}{L^2}\right)$, $\eta_2=\cO\left(\frac{1}{L}\right),\eta_3 = \cO\left(\frac{1}{L}\right)$. The iterative points $(\tilde x^{k+1},\tilde y^{k+1})$ in \emph{(ADE)} are the same as that in \emph{(EG)}, except that $\eta_2$ is replaced by $\eta_3$.
\end{example}
The original update of (ADE) algorithm is rather complex since it involves the handling of constraints. In the unconstrained case, it can be simplified to the current form, which is a mixture of (GDA) and (EG). The corresponding iteration complexity bounds are $\cO\left(\frac{L^2}{\mu^2}\ln\left(\frac{1}{\epsilon}\right)\right)$ for (GDA), and  $\cO\left(\frac{L}{\mu}\ln\left(\frac{1}{\epsilon}\right)\right)$ for both (EG) and (ADE).

\subsection{General deterministic algorithm classes without linear span structure}
Although all the reviewed first-order methods satisfy the \emph{linear span} property in the proximal algorithm class in Definition \ref{defn:Prox-AlgoClass} and the pure first-order algorithm class in Definition \ref{defn:Pure-AlgoClass}, this does not exclude the possibility of the deriving an algorithm that does not satisfy the linear span property. Therefore, we also define the general deterministic proximal algorithm class and the general deterministic pure first-order algorithm class as follows, whose iteration complexity lower bound can be generalized from their linear span counterpart through the technique of \emph{adversary rotation}. 
\begin{definition}[General proximal algorithm class]
	\label{defn:Det-Prox-AlgoClass}
	Consider the problem \eqref{prob:bilinear} in the problem class $\cB(L_{xy},\mu_x,\mu_y)$, denote $\theta = (L_{xy},\mu_x,\mu_y)$ as the corresponding problem parameters. Let algorithm $\cA$ belong to the general  proximal algorithm class. Then $\cA$ consists of a sequence of deterministic mappings $\{(\cA_x^1, \cA_y^1,\cA_u^1,\cA_v^1),(\cA_x^2, \cA_y^2,\cA_u^2,\cA_v^2)...\}$ such that the iterate sequence $\{(x^{k},y^{k})\}_{k=0,1,...}$ and the output sequence $\{(\tilde x^{k},\tilde y^{k})\}_{k=0,1,...}$ are generated by 
	\begin{equation}
	\label{eqn:Det-Prox-AlgoClass}\begin{cases}
	(x^k,\tilde x^k) := \cA_x^k\left(\theta; x^0,Ay^0,...,x^{k-1},Ay^{k-1}; \mathbf{prox}_{\gamma_k f}(u^k) \right),\\
	(y^k,\tilde y^k) := \cA_y^k\left(\theta; y^0,A^\top x^0,...,y^{k-1},A^\top x^{k-1}; \mathbf{prox}_{\sigma_k g}(v^k) \right),
	\end{cases}
	\end{equation} 
	where $u^k = \cA^k_u(\theta; x^0,Ay^0,...,x^{k-1},Ay^{k-1})$, $v^k = \cA^k_v(\theta; y^0,A^\top x^0,...,y^{k-1},A^\top x^{k-1})$, and 
	$(x^0,y^0)$ is any given initial solution. 
\end{definition}
One remark is that the input of the proximal mapping $\mathbf{prox}_{\gamma_kf}(\cdot)$ is  constructed with other inputs to the $\cA_x^k$, i.e., there could be another deterministic mapping $\cA_u^k$ to generate a vector  $u^k = \cA^k_u(\theta; x^0,Ay^0,...,x^{k-1},Ay^{k-1})$ and then $\mathbf{prox}_{\gamma_kf}(u^k)$ is passed to the mapping $\cA_x^k$. This $\cA^k_u$ does not need to be linear. The situation for $v^k$ and $\cA^k_v$ is similar. Similar to the general proximal algorithm class, the general pure first-order algorithm class is defined as follows. 
\begin{definition}[General pure first-order algorithm class]
	\label{defn:Det-Pure-AlgoClass}
	Consider the problem \eqref{prob:min-max} in the problem class $\cF(L_x, L_y, L_{xy},\mu_x,\mu_y)$, denote $\theta = (L_x, L_y, L_{xy},\mu_x,\mu_y)$ as the corresponding problem parameters. Let algorithm $\cA$ belong to the general  pure first-order algorithm class. Then $\cA$ consists of a sequence of deterministic mappings $\{\cA_x^1, \cA_y^1, \cA_x^2, \cA_y^2,...\}$ such that the iterate sequence $\{(x^{k},y^{k})\}_{k=0,1,...}$ and the output sequence $\{(\tilde x^{k},\tilde y^{k})\}_{k=0,1,...}$ are generated by 
	\begin{equation}
	\label{eqn:Det-Pure-AlgoClass}\begin{cases}
	(x^k,\tilde x^k) := \cA_x^k\left(\theta; x^0,\nabla_x F(x^0,y^0),...,x^{k-1},\nabla_x F(x^{k-1},y^{k-1}) \right),\\
	(y^k,\tilde y^k) := \cA_y^k\left(\theta; y^0,\nabla_y F(x^0,y^0),...,y^{k-1},\nabla_y F(x^{k-1},y^{k-1})\right),
	\end{cases}
	\end{equation}
	given any initial solution $(x^0,y^0)$.
\end{definition}
A remark is that the gradients $\nabla_x F(\cdot,\cdot)$ and $\nabla_y F(\cdot,\cdot)$ actually do not need to be taken on the previous iterates $\{(x^0,y^0),(x^1,y^1),...,(x^{k-1},y^{k-1})\}$. Similar to the proximal case, they can also be taken on some other $\{(u^0_{k-1},v^0_{k-1}),(u^1_{k-1},v^1_{k-1}),...,(u^{k-1}_{k-1},v^{k-1}_{k-1})\}$ that are generated by some mappings $\{(\cA_u^{k-1,0},\cA_v^{k-1,0}),(\cA_u^{k-1,1},\cA_v^{k-1,1}),...,(\cA_u^{k-1,k-1},\cA_v^{k-1,k-1})\}$. The reason that we do not consider this more general form is twofold. First, the simpler form in Definition \ref{defn:Det-Pure-AlgoClass} has already covered all the discussed algorithms. Second, this more general form actually shares the same iteration complexity lower bound, despite the technical complications involved. Therefore, in this paper we shall only include the gradients at the past iterates as the input to the algorithm. 

\section{Lower bound for proximal  algorithms}\label{sec:lower-bound-proximal}
\subsection{The worst-case instance}

Let us construct the following bilinearly coupled min-max saddle point problem:
\begin{equation}
\label{prob:worst-proximal}
\min_x\max_y F(x,y):=\frac{\mu_x}{2}\|x\|^2 + \frac{L_{xy}}{2}x^\top Ay - \frac{\mu_y}{2}\|y\|^2 - b^\top y
\end{equation}
where $b$ is a vector to be determined later, and the coupling matrix $A$ (hence $A^2$ and $A^4$) is defined as follows: %. Its polynomials $A^2$ and $A^4$ are also shown bellow for later reference.
\begin{equation}
\label{defn:A}
A = \begin{pmatrix}
& &    &     & 1\\
& &    &  1 & -1\\
& & 1 & -1&\\
& \udots &\udots & & \\
1& -1 & & &
\end{pmatrix}, \,\,
A^2 = \begin{pmatrix}
1   & -1 &            &               & \\
-1 & 2   & -1        &              & \\
& -1 &\ddots & \ddots & \\
&      &\ddots & 2           &-1\\
&      &              & -1         & 2
\end{pmatrix},
\,\,
A^4 = \begin{pmatrix}
2   & -3 &     1      &               &  & &\\
-3 & 6   & -4        &     1         & & & \\
1& -4 &6 & -4 &1 \\
&  \ddots     &\ddots & \ddots  &\ddots & \ddots & \\
& & 1& -4 &6 & -4 &1\\
& & & 1& -4 &6 & -4 \\
& & & & 1&-4&5
\end{pmatrix}.
\end{equation}
Note that $A^\top =A$ and $\|A\|_2 \leq 2$. Therefore \eqref{prob:worst-proximal} is an instance in the problem class $\cB(L_{xy},\mu_x,\mu_y)$. It is worth noting that the example \eqref{prob:worst-proximal} is the same as that in Proposition 2 of \cite{LowB-Quadratic}, which is a parallel work focused on pure first-order algorithm class. Here, we use the same example to elaborate the lower bound of the proximal methods over the general bilinear coupling class $\cB(L_{xy},\mu_x,\mu_y)$, as well as a warmup for the discussion of more complex problem class $\cF(L_{x},L_{y},L_{xy},\mu_x,\mu_y)$. 

Denote $e_i$ to be the $i$-th unit vector, which has $1$ at the $i$-th component and 0 elsewhere. Then by direct calculation, one can check that $A^2$ satisfies the following \emph{zero-chain property} (see Chapter 2 of  \cite{nesterov2018lectures}).

\begin{proposition}[Zero-chain property]
	\label{prop:zero-chain}
	For any vector $v\in\R^n$, if $v\in\mathrm{Span}\{e_i:i\leq k\}$ for some $1\leq k\leq n-1$,  then $A^2v\in\mathrm{Span}\{e_i:i\leq k+1\}$.
\end{proposition}
This means that if $v$ only has nonzero elements at the first $k$ entries, then $A^2v$ will have at most one more nonzero entry at the $(k+1)$-th position.

For problem \eqref{prob:worst-proximal}, the proximal operators in \eqref{eqn:Prox-AlgoClass} can be written explicitly:
\begin{eqnarray}
\label{defn:prox-closedForm-x}
\prox_{\gamma_i f}(\hat{x}_i-\gamma_iA\tilde{y}_i) & = & \argmin_x \frac{\mu_x}{2}\|x\|^2 + \frac{1}{2\gamma_i}\left\|x-(\hat{x}_i-\frac{\gamma_iL_{xy}}{2}A\tilde{y}_i)\right\|^2\\
& = & \frac{1}{1+\gamma_i\mu_x}\hat{x}_i - \frac{\gamma_iL_{xy}}{2(1+\gamma_i\mu_x)}A\tilde{y}_i\nonumber\\
& \in & \mathrm{Span}\{\hat{x}_i,A\tilde{y}_i\}.\nonumber
\end{eqnarray}
Similarly, for the $y$ block, we also have
\begin{eqnarray}
\label{defn:prox-closedForm-y}
\prox_{\sigma_i g}(\hat{y}_i+\sigma_iA^\top\tilde{x}_i)
%& = & \argmin_y \frac{\mu_y}{2}\|y\|^2 + b^Ty + \frac{1}{2\sigma_i}\left\|y-(\hat{y}_i+\frac{\sigma_iL_{xy}}{2}A^\top \tilde{x}_i)\right\|^2\\
=  \frac{\hat{y}_i-\sigma_i b}{1+\sigma_i\mu_y} +  \frac{\sigma_iL_{xy}}{2(1+\sigma_i\mu_x)}A\tilde{x}_i \in \mathrm{Span}\{\hat{y}_i,A\tilde{x}_i,b\}.
\end{eqnarray}
Let us assume the initial point to be $x^0 = y^0 = 0$ ($\cH_x^0 = \cH_y^0 = \{0\}$) without loss of generality.
%, otherwise we can consider a shifted problem.
Directly substituting \eqref{defn:prox-closedForm-x} and \ref{defn:prox-closedForm-y} into Definition \eqref{defn:Prox-AlgoClass} yields
\begin{equation*}
\begin{cases}
\cH_x^1 \subseteq \mathrm{Span}\{0\}\\
\cH_y^1 \subseteq \mathrm{Span}\{b\}
\end{cases}
\quad
\begin{cases}
\cH_x^2 \subseteq \mathrm{Span}\{Ab\}\\
\cH_y^2 \subseteq \mathrm{Span}\{b\}
\end{cases}
\quad
\begin{cases}
\cH_x^3 \subseteq \mathrm{Span}\{Ab\}\\
\cH_y^3 \subseteq \mathrm{Span}\{b,A^2b\}
\end{cases}
\quad
\begin{cases}
\cH_x^4 \subseteq \mathrm{Span}\{Ab, A^3b\}\\
\cH_y^4 \subseteq \mathrm{Span}\{b,A^2b\}
\end{cases}
\ldots
\end{equation*}
We formally summarize this observation below: %in the Lemma \ref{lemma:subspace-prox}.
\begin{lemma}
	\label{lemma:subspace-prox}
	For problem \eqref{prob:worst-proximal}, for any $k \in \mathbb{N}$, if the iterates are generated so that  $(x_k,y_k)\in\cH_x^k\times\cH_y^k$, with $\cH_x^k$ and $\cH_y^k$ defined by \eqref{defn:Prox-AlgoClass}, then based on \eqref{defn:prox-closedForm-x} and \eqref{defn:prox-closedForm-y} the search subspaces satisfy
	$$\cH_x^k \subseteq \begin{cases}
	\{0\}, &  k = 1\\
	\mathrm{Span}\left\{A^{2i}(Ab): 0\leq i\leq \left\lfloor \frac{k}{2}\right\rfloor -1 \right\}, & k\geq 2
	\end{cases} \quad \mbox{ and }\quad
	\cH_y^k \subseteq \mathrm{Span}\left\{A^{2i}b: 0\leq i\leq \left\lceil \frac{k}{2}\right\rceil -1 \right\}.$$
\end{lemma}

\subsection{Lower bounding the duality gap}
%To lower bound the duality gap, in this case,
Let us lower bound the dual gap, which is upper bounded by the whole duality gap. To achieve this, let us first write down the dual function of problem \eqref{prob:worst-proximal} as
\begin{equation}
\label{defn:dual-function-prox}
\Psi(y) = \min_{x} F(x,y) = -\frac{1}{2}y^\top \left(\frac{L_{xy}^2}{4\mu_x}\cdot A^2 + \mu_y\cdot I\right)y - b^\top y . 
\end{equation}
For this $\mu_y$-strongly concave dual function, we can characterize the optimal solution $y^*$ directly by its KKT condition $\nabla \Psi(y^*) = 0$. However, the exact solution $y^*$ does not have a simple and clear form, so we choose to  characterize it by an approximate solution $\hat y^*$.

\begin{lemma}[Approximate optimal solution]
	\label{lemma:apxsolu-proximal}
	Let us assign the value of $b$ as $b := -\frac{L_{xy}^2}{4\mu_x}e_1$. Denote $\alpha := \frac{4\mu_x\mu_y}{L_{xy}^2}$, and let $q = \half\left((2+\alpha) - \sqrt{(2+\alpha)^2 - 4}\right)\in(0,1)$
	be the smallest root of the quadratic equation $1 - (2+\alpha)q + q^2 = 0$. Then, an approximate optimal solution $\hat y^*$ can be constructed as
	\begin{equation}
	\hat y^*_{i} = \frac{q^i}{1-q} \quad\mbox{ for } \quad i = 1,2,...,n.
	\end{equation}
	The approximation error can be bounded by
	\begin{equation}
	\label{eqn:apxsolu-proximal}
	\|\hat y^* - y^*\| \leq  \frac{q^{n+1}}{\alpha(1-q)},
	\end{equation}
	where $\hat y^*_i$ is the $i$-th element of $\hat y^*$. Note that $q<1$ and the lower bound is dimension-independent, hence we are free to choose $n$ to make the approximation error arbitrarily small.
\end{lemma}
\begin{proof}
	First, let us substitute the value of $b$ into the KKT system $\nabla \Psi(y^*)=0$, by slight rearranging and scaling the terms, we get
	\begin{equation*}
	\left(A^2 + \frac{4\mu_x\mu_y}{L_{xy}^2} I\right)y^* = - \frac{4\mu_x}{L_{xy}^2}b.
	\end{equation*}
	Using the definition of $\alpha$ and $b$, the equation becomes
	$$(A^2+\alpha I)y^* = e_1.$$
	Substituting the formula of $A^2$ in \eqref{defn:A}, we expand the above equation as
	\begin{eqnarray}
	\label{defn:KKT-proximal}
	\begin{cases}
	(1+\alpha)y_1^* - y_2^* = 1\\
	-y_1^* + (2+\alpha)y_2^* - y_3^* = 0\\
	\qquad\qquad\quad\vdots\\
	-y_{n-2}^* + (2+\alpha)y_{n-1}^* - y_n^* = 0\\
	-y_{n-1}^* + (2+\alpha)y_n^* = 0.
	\end{cases}
	\end{eqnarray}
	By direct calculation, we can check that $\hat y^*$ satisfies the first $n-1$ equations of the KKT system \eqref{defn:KKT-proximal}. The last equation, however, is violated, but with a residual of size $q^{n+1}/(1-q)$.  In details,
	$$\begin{cases}
	(A^2 + \alpha\cdot I)\hat y^* = e_1 + \frac{q^{n+1}}{1-q}\cdot e_n\\
	(A^2 + \alpha\cdot I)y^* = e_1.
	\end{cases}$$
	This indicates that $\hat y^* - y^* = \frac{q^{n+1}}{1-q}\cdot(A^2 + \alpha I)^{-1} e_n$. Note that $\alpha^{-1}I\succeq (A^2 + \alpha I)^{-1}\succ 0$, we have the approximation error bounded by \eqref{eqn:apxsolu-proximal}.
\end{proof}

 Note that in Lemma \ref{lemma:apxsolu-proximal}, we have chosen  $b\propto e_1$. By the zero-chain property in Proposition \ref{prop:zero-chain} and Lemma \ref{lemma:subspace-prox},  we can verify that the subspaces $\cH_y^{2k-1}$ and $\cH_y^{2k}$ satisfy
\begin{equation}
\label{eqn:subspace-prox}
\cH_y^{2k-1}, \cH_y^{2k} \subseteq \mathrm{Span}\{b,A^2b,...,A^{2(k-1)}b\} =\mathrm{Span}\{e_1, e_2, ...,e_k\}.
\end{equation}
This implies that for both $y^{2k}$ and $y^{2k-1}$, the only possible nonzero elements are the first $k$ ones, which again implies that the lower bound of $\|y^{2k}-y^*\|^2$ and $\|y^{2k-1}-y^*\|^2$ will be similar. For simplicity, we only discuss this lower bound  for $y^{2k}$. The counterpart for $y^{2k-1}$ can be obtained in a similar way.  Therefore, we have the following estimations.

\begin{lemma}
	\label{lemma:lowerbound-distance-prox}
	Assume $k\leq \frac{n}{2}$ and $n\geq2\log_q\left(\frac{\alpha}{4\sqrt{2}}\right)$. Then
	\begin{equation}
	\label{eqn:lowB-dis-prox}
	\|y^{2k}-y^*\|^2\geq \frac{q^{2k}}{16}\|y^0-y^*\|^2
	\end{equation}
	where $y^0 = 0$ is the initial solution.
\end{lemma}
\begin{proof}
	By the subspace characterization \eqref{eqn:subspace-prox}, we have
	$$\|y^{2k}-\hat y^*\| \geq \sqrt{\sum^n_{j = k+1} (\hat y^*_j)^2} = \frac{q^{k}}{1-q}\sqrt{q^2 + q^4 + \cdots + q^{2(n-k)}}\geq \frac{q^k}{\sqrt{2}}\|\hat y^*\|= \frac{q^k}{\sqrt{2}}\|y^0-\hat y^*\|,$$
	where the last inequality is due to the fact that $q< 1$, $k\leq \frac{n}{2}$, and $y^0 = 0$.
 If we choose $n$ to be large enough, then $\hat y^*$ and $y^*$ can be made arbitrarily close to each other. Hence we can transform the above inequality to  \eqref{eqn:lowB-dis-prox}. More details of this derivation can be found in Appendix \ref{appendix-lowerbound-distance-prox}.
\end{proof}

Using Lemma \ref{lemma:lowerbound-distance-prox} and \eqref{defn:primal-gap-LB}, it is then straightforward to lower bound the duality gap by
$$\Delta(x^{2k},y^{2k})\geq q^{2k}\cdot\frac{\mu_y\|y^*-y^0\|^2}{32}.$$
Summarizing, below we present our first main result.
\begin{theorem}
	\label{thm:proximal}
 Let the positive parameters $\mu_x,\mu_y>0$ and $L_{xy}> 0$ be given. For any integer $k$, there exists a problem instance from $\cB(L_{xy},\mu_x,\mu_y)$ of form \eqref{prob:worst-proximal}, with $n\geq \max\left\{2\log_q\left(\frac{\mu_x\mu_y}{\sqrt{2}L_{xy}^2}\right), 4k\right\}$, where $A\in\mathbb{R}^{n\times n}$ as defined in \eqref{defn:A}, and $b = -\frac{L_{xy}^2}{4\mu_x}e_1$. For such a problem instance, any approximate solution $(\tilde x^k, \tilde y^k)\in\cH_x^k\times\cH_y^k$ generated by the proximal algorithm class under the linear span assumption \eqref{eqn:Prox-AlgoClass}  satisfies
	\begin{eqnarray}
	\max_y F(\tilde x^{k},y) - \min_{x} F(x, \tilde y^k)\geq q^{k}\cdot\frac{\mu_y\|y^*-y^0\|^2}{32} \quad\mbox{ and }\quad \|\tilde y^k-y^*\|^2\geq q^{k}\cdot\frac{\|y^*-y^0\|^2}{16},
	\end{eqnarray}
	where $q = 1+\frac{2\mu_x\mu_y}{L_{xy}^2} - 2\sqrt{\left(\frac{\mu_x\mu_y}{L_{xy}^2}\right)^2 + \frac{\mu_x\mu_y}{L_{xy}^2}}$.
\end{theorem}
\begin{proposition}
	\label{remark:complexity-proximal}
	Under the same set of assumptions of Theorem \ref{thm:proximal}, if we require the duality gap to be bounded by $\epsilon$, the number of iterations needed is at least
	\begin{equation}
	\label{eqn:lowerbound-proximal-1stAlgo}
	k\geq \ln\left(\frac{\mu_y\|y^*-y^0\|^2}{32\epsilon}\right)/\ln(q^{-1}) = \Omega\left(\sqrt{\frac{L_{xy}^2}{\mu_x\mu_y}}\cdot\ln\left(\frac{1}{\epsilon}\right)\right).
	\end{equation}
\end{proposition}
%The detailed calculation of this lemma is moved to
The proof of Proposition~\ref{remark:complexity-proximal} is in Appendix \ref{appendix:r-complexity-proximal}.

\subsection{The general proximal algorithm class}
Note that Theorem \ref{thm:proximal} is derived for the proximal algorithm class with the linear span assumption, in this section, we will apply the orthogonal invariance technique, introduced in \cite{nemirovsky1992information}, to generalize the result of Theorem \ref{thm:proximal} to the general proximal algorithm class without the linear span assumption.

Consider the bilinear problem class $\cB(L_{xy},\mu_x,\mu_y)$ and the corresponding worst case problem \eqref{prob:worst-proximal} with $F(x,y) = \frac{\mu_x}{2}\|x\|^2 + \frac{L_{xy}}{2}x^\top A y - \frac{\mu_y}{2}\|y\|^2 - b^\top y$,
where $A$ and $b$ are defined in accordance with Theorem \ref{thm:proximal}. We define the orthogonally rotated problem as 
\begin{equation}
\label{prob:rotate}
\min_x\max_y F_{U,V}(x,y) : = F(Ux,Vy) = \frac{\mu_x}{2}\|x\|^2 + \frac{L_{xy}}{2}x^\top U^\top AV y - \frac{\mu_y}{2}\|y\|^2 - b^\top Vy,
\end{equation}
where $U,V$ are two orthogonal matrices. Therefore, it is clear that $F_{U,V}\in\cB(L_{xy},\mu_x,\mu_y)$. Let $(x^*, y^*)$ be the saddle point of $F(x,y)$, then it is clear that the saddle point of $F_{U,V}(x,y)$ is $(\bar x^*,\bar y^*) = (U^\top  x^*,V^\top y^*).$ Consequently, the lower bound for the general proximal algorithm class is characterized by the following theorem. 
\begin{theorem}
\label{theorem:gen-prox}
Let $\cA$ be any algorithm from the general proximal algorithm class decribed in Definition \ref{defn:Det-Prox-AlgoClass}. We assume the dimension $n$ is sufficiently large for simplicity. For any integer $k$, then there exist orthogonal matrices $U,V$ s.t. $F_{U,V}\in\cB(L_{xy},\mu_x,\mu_y)$, when applying $\cA$ to $F_{U,V}$ with initial solution $(x^0,y^0) = (0,0)$, the iterates and output satisfies
$$\{(x^0,y^0),...,(x^k,y^k)\}\subseteq U^\top\cH_x^{4k-1}\times V^\top\cH_y^{4k-1}\qquad\mbox{and}\qquad(\tilde x^k,\tilde y^k)\in U^\top\cH_x^{4k+1}\times V^\top\cH_y^{4k+1},$$
where $\cH_x^i,\cH_y^i$ are defined by Lemma \ref{lemma:subspace-prox}. Consequenty, by Theorem \ref{thm:proximal}, 
\begin{eqnarray*}
\|\tilde y^k-V^\top  y^*\|^2\geq \frac{q^{4k+2}}{16}\|y^*-y^0\|^2
\end{eqnarray*}
where $q$ is given in Theorem \ref{thm:proximal}. As a result, it takes $\Omega\left(\sqrt{\frac{L^2_{xy}}{\mu_x\mu_y}}\cdot\log\left(\frac{1}{\epsilon}\right)\right)$ iterations to output a solution with $O(\epsilon)$ duality gap.
\end{theorem}
For the proof of this theorem, we only need to construct the orthogonal matrices $U,V$ such that when the algorithm $\cA$ is applied to $F_{U,V}$, the subspace inclusion argument $\{(x^0,y^0),...,(x^k,y^k)\}\subseteq U^\top\cH_x^{4k-1}\times V^\top\cH_y^{4k-1}$ and $(\tilde x^k,\tilde y^k)\in U^\top\cH_x^{4k+1}\times V^\top\cH_y^{4k+1}$ holds. As a result, 
\begin{eqnarray*}
	\|\tilde y^k-V^\top  y^*\|^2 = \|V\tilde y^k- y^*\|^2\geq\min_{y\in\cH_y^{4k+1}}\|y-y^*\|^2\geq \frac{q^{4k+2}}{16}\|y^*-y^0\|^2.
\end{eqnarray*}
With this argument, the latter results follow directly from the discussion of Theorem \ref{thm:proximal}. The proof of this theorem is presented in Appendix \ref{appdx:gen-prox}.
\subsection{Tightness of the bound}
We claim the tightness of the derived lower bound by the following remark.
\begin{remark}[Tightness of the bound] Consider the algorithm defined in Example \ref{algo:ChambolleNPock}, from \cite{chambolle2011first,chambolle2016ergodic}.  The achieved upper complexity bound is $\cO\left(\sqrt{\frac{L_{xy}^2}{\mu_x\mu_y}}\cdot\ln\left(\frac{1}{\epsilon}\right)\right)$, and it matches our lower bound. This means that our lower bound \eqref{eqn:lowerbound-proximal-1stAlgo} is tight and the algorithm defined in Example \ref{algo:ChambolleNPock} is an optimal algorithm in the proximal algorithm class in Definition \ref{defn:Prox-AlgoClass}.
\end{remark}

\section{Lower bound for pure first-order algorithms} \label{sec:lower-bound-pure}
\subsection{The worst-case instance}
In this section, we consider the lower complexity bound for the pure first-order method without any proximal oracle. In this case, only the gradient information can be used to construct the iterates and produce the approximate solution output. Similar as before, we still consider the bilinearly coupled problems: 
\begin{equation}
\label{prob:worst-pure}
\min_x\max_y F(x,y):=\half x^\top (B_xA^2 + \mu_xI)x + \frac{L_{xy}}{2}x^\top Ay - \half y^\top (B_yA^2 + \mu_yI)y - b^\top y
\end{equation}
where $b$ is a vector whose value will be determined later. The coefficients $B_x: = \frac{L_x-\mu_x}{4}, B_y:= \frac{L_y-\mu_y}{4}$ and the coupleing matrix $A$ is defined by \eqref{defn:A}. Note that $\|A\|_2\leq 2$ and $\|A\|_2^2\leq 4$, we can check that problem \eqref{prob:worst-pure} is an instance from the problem class $\cF(L_x,L_y,L_{xy},\mu_x,\mu_y)$. This time the subspaces $\cH_x^k$'s and $\cH_y^k$'s are generated by the following gradients:
\begin{equation*}
\begin{cases}
\nabla_x F(x,y) = (B_xA^2 + \mu_xI)x + \frac{L_{xy}}{2}Ay,\\
\nabla_yF(x,y) = -(B_yA^2 + \mu_yI)y + \frac{L_{xy}}{2}Ax - b.
\end{cases}
\end{equation*}
Following Definition \ref{defn:Pure-AlgoClass}, by letting $x^0 = y^0 = 0$ we have
\begin{equation*}
\begin{cases}
\cH_x^1 \subseteq \mathrm{Span}\{0\}\\
\cH_y^1 \subseteq \mathrm{Span}\{b\}
\end{cases}
\quad
\begin{cases}
\cH_x^2 \subseteq \mathrm{Span}\{Ab\}\\
\cH_y^2 \subseteq \mathrm{Span}\{b, A^2b\}
\end{cases}
\quad
\begin{cases}
\cH_x^3 \subseteq \mathrm{Span}\{Ab,A^2(Ab)\}\\
\cH_y^3 \subseteq \mathrm{Span}\{b,A^2b,A^4b\}
\end{cases}
\quad\ldots
\end{equation*}
By induction, we get the general structure of these subspaces.
\begin{lemma}
	\label{lemma:subspace-pure}
	For problem \eqref{prob:worst-pure} and  for any $k \in \mathbb{N}$, if the iterates are generated so that $(x_k,y_k)\in\cH_x^k\times\cH_y^k$, with $\cH_x^k$ and $\cH_y^k$ defined by \eqref{defn:Pure-AlgoClass}, then we have
	$$\cH_x^k \subseteq \begin{cases}
	\{0\}, &  k = 1\\
	\mathrm{Span}\left\{A^{2i}(Ab): 0\leq i\leq k-2 \right\}, & k\geq 2
	\end{cases} \quad \mbox{ and }\quad
	\cH_y^k \subseteq \mathrm{Span}\left\{A^{2i}b: 0\leq i\leq k -1 \right\}.$$
\end{lemma}

Different from the discussion of last section, this time it is more convenient to deal with the primal function instead of the dual one. By partially maximizing over $y$ we have
\[
\Phi(x):=\max_y F(x,y) = \half x^\top (B_xA^2+\mu_xI)x + \frac{L_{xy}^2}{8}\left(Ax-\frac{2b}{L_{xy}}\right)^\top (B_yA^2+\mu_yI)^{-1}\left(Ax-\frac{2b}{L_{xy}}\right),
\]
which is $\mu_x$-strongly convex. Therefore, the primal optimal solution $x^*$ is completely characterized by the optimality condition $\nabla \Phi(x^*) = 0$. However, the solution of this system cannot be computed exactly. Instead, we shall construct an approximate solution $\hat x^*$ to the exact solution $x^*$.

\begin{lemma}[Root estimation]\label{lemma:root-estimation} Consider a quartic equation
	\begin{equation}
	\label{eqn:4EQ}
	1 -(4+\alpha)x  + (6+2\alpha + \beta)x^2 - (4+\alpha)x^3  + x^4 = 0,
	\end{equation}
	where the constants are given by
	\begin{equation}
	\label{defn:constants-pure}
	\alpha = \frac{L_{xy}^2}{4B_xB_y} + \frac{\mu_x}{B_x} + \frac{\mu_y}{B_y},\qquad \beta = \frac{\mu_x\mu_y}{B_xB_y}.
	\end{equation}
	As long as $L_x>\mu_x>0,$ and $L_y>\mu_y>0$. Then the constants  $0<\alpha,\beta<+\infty$ are well-defined positive real numbers. For this quartic equation, it has a real root $x=q$ satisfying 
	\begin{equation}
	\label{eqn:interval-pure}
	1-\left(\half +\frac{1}{2\sqrt{2}}\sqrt{\frac{L_x}{\mu_x} + \frac{L_{xy}^2}{\mu_x\mu_y} + \frac{L_y}{\mu_y}}\right)^{-1}<q<1 - \left(\half +\frac{1}{2}\sqrt{\frac{L_{xy}^2}{\mu_x\mu_y} + \frac{L_x}{\mu_x} + \frac{L_y}{\mu_y}-1}\right)^{-1} .
	\end{equation}
\end{lemma}
The proof of this lemma is presented in Appendix \ref{appendix-root-estimation}. With this lemma, we can construct the approximate solution $\hat x^*$ as follows.

\begin{lemma}[Approximate optimal solution]
	\label{lemma:apxsolu-pure}
	Let $\alpha,\beta$ be defined according to \eqref{defn:constants-pure}, let $q$ be a real root of quartic equation \eqref{eqn:4EQ} satisfying \eqref{eqn:interval-pure}. Let us define a vector $\hat b$ with elements given by
	\begin{equation}
	\label{defn:hat-b}
	\hat b_1 := (2+\alpha + \beta)q - (3+\alpha)q^2+q^3,\qquad \hat b_2 :=q-1, \quad \mbox{ and } \quad \hat b_k = 0, \mbox{ for } 3\leq k \leq n,
	\end{equation}
	 and then assign $b = \frac{2B_xB_y}{L_{xy}}A^{-1}\hat b$. 
	Then an approximate solution $\hat x^*$ is constructed as
	\begin{equation}
	\hat x^*_{i} = q^i \quad\mbox{ for } \quad i = 1,2,...,n.
	\end{equation}
	The approximation error can be bounded by
	\begin{equation}
	\label{eqn:apxsolu-pure}
	\|\hat x^*-x^*\|\leq \frac{7+\alpha}{\beta}\cdot q^n.
	\end{equation}
	Note that $q<1$ and the lower bound is dimension-independent, hence we are free to choose $n$ to make the approximation error arbitrarily small.
\end{lemma}
The proof of this lemma is parallel to that of Lemma \ref{lemma:apxsolu-proximal}, but is more involved; the detailed proof is in Appendix \ref{appendix-apxsolu-pure}.

Note that in this case, the vector $Ab\propto \hat b\subset\mathrm{Span}\{e_1,e_2\}$. By the zero-chain property in Proposition \ref{prop:zero-chain}, the subspace $\cH_x^k$ described in Lemma \ref{lemma:subspace-pure} can be calculated by induction
\begin{equation}
\label{eqn:subspace-pure}
\cH_x^k\subset\mathrm{Span}\{e_1,e_2,...,e_k\} \quad \mbox{ for } \quad k\geq 2.
\end{equation}
Parallel to Lemma \ref{lemma:lowerbound-distance-prox}, we have the following lemma, whose proof is in Appendix \ref{appendix-lowerbound-distance-pure}.
\begin{lemma}
	\label{lemma:lowerbound-distance-pure}
	Assume $k\leq \frac{n}{2}$ and $n\geq 2\log_q\left(\frac{\beta}{4\sqrt{2}(7+\alpha)}\right)+2$. Then
	\begin{equation}
	\label{eqn:dist-lower-pure}
	\|x^k-x^*\|^2\geq \frac{q^{2k}}{16}\|x^*-x^0\|^2
	\end{equation}
	where $x^0 = 0$ is the initial solution.
\end{lemma}
Consequently, the duality gap is lower bounded by
$$\Delta(x^k,y^k)\geq \frac{\mu_x}{2}\|x^k-x^*\|^2\geq q^{2k}\cdot\frac{\mu_x\|x^0-x^*\|^2}{32}.$$
Summarizing, we present our second main result in the following theorem. %the above result in the following Theorem.
\begin{theorem}
	\label{thm:pure}
Let positive parameters $\mu_x,\, \mu_y>0$ and $L_x>\mu_x,\, L_y>\mu_y,\, L_{xy} > 0$ be given.  For any integer $k$, there exists a problem instance in $\cF(L_x,L_y,L_{xy},\mu_x,\nu_y)$ of form \eqref{prob:worst-pure}, with $n\geq \max\left\{2\log_q\left(\frac{7+\alpha}{\beta}\right), 2k\right\}$, the constants $\alpha,\beta$ as in \eqref{defn:constants-pure}, the matrix $A\in\mathbb{R}^{n\times n}$ as in \eqref{defn:A}, the vector $b = \frac{2B_xB_y}{L_{xy}}A^{-1}\hat b$ where $\hat b$ as in \eqref{defn:hat-b}. For this problem, any approximate solution $(\tilde x^k, \tilde y^k)\in\cH_x^k\times\cH_y^k$ generated by first-order algorithm  class \eqref{eqn:Pure-AlgoClass} satisfies
	\begin{eqnarray}
	\max_y F(\tilde x^{k},y) - \min_{x} F(x, \tilde y^k)\geq q^{2k}\cdot\frac{\mu_x\|x^*-x^0\|^2}{32} \quad\mbox{ and }\quad \|\tilde x^k-x^*\|^2\geq q^{2k}\cdot\frac{\|x^*-x^0\|^2}{16},
	\end{eqnarray}
where %{\small$q \in \left(1-\left(\half +\frac{1}{2\sqrt{2}}\sqrt{\frac{L_x}{\mu_x} + \frac{L_{xy}^2}{\mu_x\mu_y} + \frac{L_y}{\mu_y}}\right)^{-1},1\right)$} 
$q$ satisfying \eqref{eqn:interval-pure} is a root of the quartic equation \eqref{eqn:4EQ}.
\end{theorem}

\begin{remark}
	\label{remark:complexity-pure}
	As a result, if we require the duality gap to be bounded by $\epsilon$, then the number of iterations needed is at least
	\begin{equation}
	\label{eqn:lowerbound-pure-1stAlgo}
	k\geq \half\ln\left(\frac{\mu_x\|x^*-x^0\|^2}{32\epsilon}\right)/\ln(q^{-1}) = \Omega\left(\sqrt{\frac{L_x}{\mu_x} + \frac{L_{xy}^2}{\mu_x\mu_y} + \frac{L_y}{\mu_y}}\cdot\ln\left(\frac{1}{\epsilon}\right)\right).
	\end{equation}
\end{remark}

\subsection{The general pure first-order algorithm class}
Consider the problem class $\cF(L_x,L_y,L_{xy},\mu_x,\mu_y)$, we define the orthogonally rotated problem as 
{\small
\begin{equation}
\label{prob:rotate'}
\min_x\max_y F_{U,V}(x,y) : =  \half x^\top (B_xU^\top\!\!A^2U + \mu_xI)x + \frac{L_{xy}}{2}x^\top U^\top\!\!AVy - \half y^\top (B_yV^\top\!\!A^2V + \mu_yI)y - b^\top V^\top\!y
\end{equation}}
$\!\!$where $A,b,B_x,B_y$ are defined in accordance with Theorem \ref{thm:pure} and  $U,V$ are two orthogonal matrices. Therefore, it is clear that $F_{U,V}\in\cF(L_x,L_y,L_{xy},\mu_x,\mu_y)$. Let $( x^*,y^*)$ be the saddle point of $F(x,y)$, then it is clear that the saddle point of $F_{U,V}(x,y)$ is $(\bar x^*,\bar y^*) = (U^\top \bar x^*,V^\top\bar y^*).$ Consequently, the lower bound for the general proximal algorithm class is characterized by the following theorem. 
	\begin{theorem}
		\label{theorem:gen-pure}
		Let $\cA$ be any algorithm from the general pure first-order algorithm class described in Definition \ref{defn:Det-Pure-AlgoClass}. We assume the dimension $n$ is sufficiently large for simplicity. For any integer $k$, then there exist orthogonal matrices $U,V$ s.t. $F_{U,V}\in\cF(L_x,L_y,L_{xy},\mu_x,\mu_y)$, when applying $\cA$ to $F_{U,V}$ with initial solution $(x^0,y^0) = (0,0)$, the iterates and output satisfies
		$$\{(x^0,y^0),...,(x^k,y^k)\}\subseteq U^\top\cH_x^{2k}\times V^\top\cH_y^{2k}\qquad\mbox{and}\qquad(\tilde x^k,\tilde y^k)\in U^\top\cH_x^{2k+1}\times V^\top\cH_y^{2k+1},$$
		where $\cH_x^i,\cH_y^i$ are defined by Lemma \ref{lemma:subspace-pure}. Consequenty, by Theorem \ref{thm:proximal}, 
		\begin{eqnarray*}
			\|\tilde x^k-U^\top x^*\|^2 \geq \frac{q^{4k+2}}{16}\|x^*-x^0\|^2
		\end{eqnarray*}
		where $q$ is defined in Theorem \ref{thm:pure}. As a result, it takes $\Omega\left(\sqrt{\frac{L_x}{\mu_x} + \frac{L_{xy}^2}{\mu_x\mu_y} + \frac{L_y}{\mu_y}}\cdot\ln\left(\frac{1}{\epsilon}\right)\right)$ iterations to output a solution with $O(\epsilon)$ duality gap.
	\end{theorem}
	The proof of this theorem is completely parallel to that of Theorem \ref{theorem:gen-prox}. We only need to construct the orthogonal matrices $U,V$ such that $\{(x^0,y^0),...,(x^k,y^k)\}\subseteq U^\top\cH_x^{2k}\times V^\top\cH_y^{2k}$ and $(\tilde x^k,\tilde y^k)\in U^\top\cH_x^{2k+1}\times V^\top\cH_y^{2k+1}$ hold, whose proof follows exactly the same proof procedure of Theorem \ref{theorem:gen-prox}. Then argue that 
	\begin{eqnarray*}
		\|\tilde x^k-U^\top x^*\|^2 = \|U\tilde x^k- x^*\|^2\geq\min_{x\in\cH_x^{2k+1}}\|x-x^*\|^2\geq \frac{q^{4k+2}}{16}\|x^*-x^0\|^2.
	\end{eqnarray*}
 The latter results follow Theorem \ref{thm:pure} and we omit the proof. 

\subsection{Tightness of the bound}
In this section, we discuss the tightness of this bound. Currently, to the best of our knowledge, there does not exist a pure first-order algorithm that can achieve the lower complexity bound provided in \eqref{eqn:lowerbound-pure-1stAlgo}. Therefore, whether an optimal algorithm exists that can match this bound or the bound can be further improved remains an open problem. However, we shall see below that %argue the tightness of 
\eqref{eqn:lowerbound-pure-1stAlgo} under several special parameter regimes is indeed a tight bound.

\textbf{Case 1: $\cF(L_x,L_y, L_{xy},\mu_x,\mu_y)$}   \,\,\,
For this general class, define $L = \max\{L_x,L_y,L_{xy}\}$. A near optimal upper bound of 
$$\cO\left(\sqrt{\frac{L_x}{\mu_x} + \frac{L\cdot L_{xy}}{\mu_x\mu_y} + \frac{L_y}{\mu_y}}\cdot\ln\left(\frac{1}{\epsilon}\right)\right)$$
is obtained in \cite{wang2020improved}, which almost matches our lower bound. 

\textbf{Case 2: $\cF(L_x,L_y, 0,\mu_x,\mu_y)$}   \,\,\,
 In this case $L_{xy} = 0$, meaning that variables $x$ and $y$ are decoupled.  Problem \eqref{prob:min-max} becomes two independent convex problems with condition numbers $\frac{L_x}{\mu_x}$ and $\frac{L_y}{\mu_y}$ respectively. In this case \eqref{eqn:lowerbound-pure-1stAlgo} is reduced to
 $$\Omega\left(\sqrt{\frac{L_x}{\mu_x}}\ln\left(\frac{1}{\epsilon}\right) + \sqrt{\frac{L_y}{\mu_y}}\ln\left(\frac{1}{\epsilon}\right)\right).$$
 This is matched by running two independent Nesterov's accelerated gradient methods \cite{nesterov2018lectures}.

\textbf{Case 3: $\cF(L,L,L,\mu,\mu)$} \,\,\, In this case $L_x = L_y = L_{xy} = L$, $\mu_x = \mu_y = \mu$. Then \eqref{eqn:lowerbound-pure-1stAlgo} is reduced to
\[
\Omega\left(\frac{L}{\mu}\ln\left(\frac{1}{\epsilon}\right)\right) . 
\]
The extra-gradient algorithm (EG) and the accelerated dual extrapolation algorithm (ADE) introduced in Example \ref{algo:single-loop} have achieved this bound; see e.g. \cite{nesterov2006solving,mokhtari2019unified}.

\textbf{Case 4: $\cF(L_x,\cO(1)\cdot\mu_y,L_{xy},\mu_x,\mu_y)$} \,\,\, In this case $L_y = \cO(1)\cdot\mu_y$, meaning that one side of the problem is easy to solve. Then,  \eqref{eqn:lowerbound-pure-1stAlgo} is reduced to
$$\Omega\left(\sqrt{\frac{L_x}{\mu_x} + \frac{L_{xy}^2}{\mu_x\mu_y}}\cdot\ln\left(\frac{1}{\epsilon}\right)\right).$$
For the double loop algorithm defined in Example \ref{algo:double-loop}, when we set the inner loop iteration to be {\small$T_2 = \cO\left(\sqrt{\frac{L_y}{\mu_y}}\ln\left(\frac{1}{\epsilon}\right)\right) = \cO\left(\ln\left(\frac{1}{\epsilon}\right)\right)$}, and the outer loop iteration to be {\small$T_1 = \cO\left(\sqrt{\frac{L_{\Phi,x}}{\mu_x}}\ln\left(\frac{1}{\epsilon}\right)\right)$$ = \cO\left(\sqrt{\frac{L_x}{\mu_x} + \frac{L_{xy}^2}{\mu_x\mu_y}}\cdot\ln\left(\frac{1}{\epsilon}\right)\right)$}. Then, an upper bound of
\[
T_1T_2 = \cO\left(\sqrt{\frac{L_x}{\mu_x} + \frac{L_{xy}^2}{\mu_x\mu_y}}\cdot\ln^2\left(\frac{1}{\epsilon}\right)\right)
\]
can be guaranteed. It is tight up to a logarithmic factor.

\textbf{Case 5: $\cF(L,L,L,\mu_x,\mu_x)$} \,\,\, In this case $L_x = L_y = L_{xy} = L$. Then \eqref{eqn:lowerbound-pure-1stAlgo} is reduced to
\[
\Omega\left(\sqrt{\frac{L^2}{\mu_x\mu_y}}\ln\left(\frac{1}{\epsilon}\right)\right) . 
\]
This bound has been achieved by \cite{lin2020near} up to a logarithmic factor.

\section{Reduction to lower bounds for convex-concave problems.}
Note that in the previous sections, we consider the strongly-convex and strongly-concave problem classes $\cB(L_{xy},\mu_x,\mu_y)$ and $\cF(L_{x},L_{y},L_{xy},\mu_x,\mu_y)$, where $\mu_x>0$ and $\mu_y>0$. In this section, we show how our iteration complexity lower bounds provided in Theorem \ref{thm:proximal} and Theorem \ref{thm:pure} can be reduced to the problem classes with $\mu_x = \mu_y = 0$. Similar reduction can also be done for the case where $\mu_x>0, \mu_y = 0$, but is omitted in this paper. 
\subsection{Lower bound for pure first-order algorithm class}
Unlike the strongly-convex and strongly-concave saddle point problems, the saddle point of the general convex-concave problem may not always exist. Therefore, we define a new problem class with bounded saddle point solution as follows. 
\begin{definition}(Problem class $\cF_0(L_x,L_y,L_{xy},R_x,R_y)$)
	\label{definition:F-cc}
	We say a function $F$ belongs to the class $\cF_0(L_x,L_y,L_{xy},R_x,R_y)$ as long as: (i). $F\in\cF(L_x,L_y,L_{xy},0,0).$ (ii). The solution to $(x^*,y^*) = \argmin_x\argmax_y F(x,y)$ exists, and $\|x^*\|\leq R_x$, $\|y^*\|\leq R_y$.
\end{definition}
For this problem class, we have the following lower bound result, as a corollary of Theorem \ref{theorem:gen-pure}. 
\begin{corollary}
	\label{corollary:pure-cc}
	Consider applying the general first-order algorithm class defined by \eqref{defn:Pure-AlgoClass} to the problem class $\cF_0(L_x,L_y,L_{xy},R_x,R_y)$. For any $\epsilon>0$, there exists a problem instance $F_\epsilon(x,y)\in\cF_0(L_x,L_y,L_{xy},R_x,R_y)$, such that 
	\begin{eqnarray}
	\label{cor:pure-cc}
	\Omega\left(\,\,\sqrt{\frac{L_xR_x^2}{\epsilon}}+\frac{L_{xy}R_xR_y}{\epsilon} + \sqrt{\frac{L_yR_y^2}{\epsilon}}\,\,\right)
	\end{eqnarray}	
	iterations are required to reduce the duality gap to $\epsilon$.
\end{corollary}
\begin{proof} 
	We start the reduction by the following scaling argument. First, for any $\epsilon>0$, let $\hat F_\epsilon\in\cF(L_x, L_y, L_{xy},\mu_x,\mu_y)$ be the worst-case instance described by Theorem \ref{theorem:gen-pure}. For our purpose, we choose
	$$\mu_x = 64\epsilon/R_x^2\qquad\mbox{and}\qquad\mu_y = 64\epsilon/R_y^2.$$ 
	Then by direct computation, we know that the following scaled problem satisfies
	$$F_\epsilon(x,y): = a \hat F_\epsilon(cx,dy)\in\cF(ac^2 L_x,ad^2 L_y,acd L_{xy},ac^2 \mu_x,ad^2 \mu_y).$$
	We skip the parameter $b$ since it is already used in the construction of the worst case instance $\hat F_\epsilon$. Denote $(\hat x^*,\hat y^*) = \min_{x}\max_y \hat F_\epsilon(x,y)$ and $(x^*,y^*) = \min_{x}\max_y F_\epsilon(x,y)$. Let us set 
	$$c = \frac{\|\hat x^*\|}{R_x}, \quad d = \frac{\|\hat y^*\|}{R_y},\quad\mbox{and}\quad a = \min\{c^{-2},d^{-2}\}.$$
	Then we have $x^* = \frac{R_x\hat x^*}{\|\hat x^*\|}$ and $y^* = \frac{R_y\hat y^*}{\|\hat y^*\|}$, and the Lipschitz constants of $F_\epsilon$ satisfy that
	$ac^2 L_x\leq L_x$, $ad^2 L_y\leq L_y$, and $acd L_{xy}\leq L_{xy}$. 
	Therefore, we know 
	$$F_\epsilon\in\cF(ac^2 L_x,ad^2 L_y,acd L_{xy},ac^2 \mu_x,ad^2 \mu_y)\cap\cF_0(L_x,L_y,L_{xy},R_x,R_y).$$
    Note that  purely scaling the variables and the function does not change the worst-case nature of this problem. In other words, $F_\epsilon$ is still the worst-case problem instance of the function class $\cF(ac^2 L_x,ad^2 L_y,acd L_{xy},ac^2 \mu_x,ad^2 \mu_y)$ and the lower bound of Theorem \ref{thm:pure} is valid for this specific instance.  Therefore, to get the duality gap less than or equal to $\epsilon$, the number of iteration $k$ is lower bounded by 
    \begin{eqnarray}
    \label{cor:pure-cc-1}
    k & \geq &  \Omega\left(\sqrt{\frac{ac^2L_x}{ac^2\mu_x} + \frac{a^2c^2d^2L_{xy}^2}{ac^2\mu_x\cdot ad^2\mu_y} + \frac{ad^2L_y}{ad^2 \mu_y}}\cdot\ln\left(\frac{ac^2\mu_x\|x^*-x^0\|^2}{32\epsilon}\right)\right)\\
    & \overset{(i)}{=} & \Omega\Bigg(\bigg(\sqrt{\frac{L_xR_x^2}{\epsilon}} + \frac{L_{xy}R_xR_y}{\epsilon}+\sqrt{\frac{L_yR_y^2}{\epsilon}}\bigg)\cdot\ln\left(2ac^2\right)\Bigg)\nonumber
    \end{eqnarray}
    where (i) is because $x^0 = 0$, $\|x^*\| = R_x$, and $\mu_x = 64\epsilon/R_x^2$. Therefore, as long as we can show that $\ln\left(2ac^2\right) \geq \Omega(1)$, then the corollary is proved. However, since the details are rather technical, we shall provide a proof of $\ln\left(2ac^2\right) \geq \ln 2 $ in Appendix \ref{appdx:Omega-1}.
\end{proof}
As a remark, by setting $L_y = 0$, the lower bound \eqref{cor:pure-cc} implies the result derived in \cite{YYXu}. When $L_x = L_y = L_{xy} = L$, the lower bound \eqref{cor:pure-cc} implies the result derived in \cite{nemirovsky1992information}. The reduction for the general pure first-order algorithm class defined by \eqref{defn:Det-Pure-AlgoClass} without the linear span assumption can also be done in a similar manner and is omitted for succinctness. 

\subsection{Lower bound for proximal algorithm class}
Like Definition \ref{definition:F-cc}, we define a new bilinear problem class with bounded saddle point solution as follows. 
\begin{definition}(Problem class $\cB_0(L_{xy},R_x,R_y)$)
	We say a function $F$ belongs to the function class $\cB_0(L_{xy},R_x,R_y)$ as long as: (i). $F\in\cB(L_{xy},0,0).$ (ii). Solution $(x^*,y^*) = \argmin_x\argmax_y F(x,y)$ exists, and $\|x^*\|\leq R_x$, $\|y^*\|\leq R_y$.
\end{definition}
For this problem class, we have the following lower bound result, as a corollary of Theorem \ref{theorem:gen-prox}. 
\begin{corollary}
	\label{corollary:prox-cc}
	Consider applying the proximal algorithm class defined by \eqref{defn:Det-Prox-AlgoClass} to the problem class $\cB_0(L_{xy},R_x,R_y)$. For any $\epsilon>0$, there exists an instance $F_\epsilon(x,y)\in\cB_0(L_{xy},R_x,R_y)$, such that 
	\begin{eqnarray}
	\label{cor:prox-cc}
	\Omega\left(\,\,\frac{L_{xy}R_xR_y}{\epsilon} \,\,\right)
	\end{eqnarray}	
	iterations are required to reduce the duality gap to  $\epsilon$.
\end{corollary}
\begin{remark}
	The lower bound in Corollary \ref{corollary:pure-cc} is tight. An optimal algorithm is derived in \cite{chambolle2011first,chambolle2016ergodic}. 
\end{remark}
The reduction can be done in a similar way as in Corollary \ref{corollary:pure-cc}, but is much simpler. The details are omitted here.

%\textbf{Case 4: $\cF(L,L,L,\mu_x,\mu_y)$} \,\,\, In this case $L_x = L_y = L_{xy} = L$, but $\mu_x$ and $\mu_y$ may still be imbalanced. In this case, \eqref{eqn:lowerbound-pure-1stAlgo} reduces to
%\[
%\Omega\left(\sqrt{\frac{L^2}{\mu_x\mu_y}}\ln\left(\frac{1}{\epsilon}\right)\right).
%\]
%Through private communication, a working paper of Tianyi Lin \& Chi Jin, et al. mathes this bound.

\section{Conclusion}\label{sec:conclusion}
In this paper, we establish the lower complexity bound for the first-order methods in solving strongly convex and strongly concave saddle point problems. Different from existing results, we discuss the problem in the most general parameter regime. For the bilinear coupling problem class $\cB(L_{xy},\mu_x,\mu_y)$ and for both proximal algorithm class \eqref{eqn:Prox-AlgoClass} with linear span assumption and the general proximal algorithm class \eqref{eqn:Det-Prox-AlgoClass} without linear span assumption, a tight lower bound is  established. For general coupling problem class  $\cF(L_x,L_y,L_{xy},\mu_x,\mu_y)$ and for both the pure first-order algorithm class \eqref{eqn:Pure-AlgoClass} with linear span assumption and the general pure first-order algorithm class \eqref{eqn:Det-Pure-AlgoClass} without the linear span assumption, a lower bound has been established. Under various special parameter regimes, tight upper-bounds can be developed. In the most general setting of the min-max framework, a near optimal algorithm has been discovered, while the optimal algorithm that exactly matches the lower bound has yet to be discovered. Finally, we also show that our result implies several exisiting lower bounds for general convex-concave problems through proper scaling of the worst-case instance, which indicates the generality of our results.
%Though we tend to believe the tightness of the lower bound, this still remains an open problem.

\textbf{Acknowledgement}. We thank the two anonymous reviewers for their insightful suggestions on orthogonal invariance argument for breaking the linear span assumption and the suggestion on applying scaling to obtain lower bounds for general convex-concave problems.

\bibliography{Saddle}

\begin{thebibliography}{10}

\bibitem{abadeh2015distributionally}
S.S. Abadeh, P.M. Esfahani, and D.~Kuhn.
\newblock Distributionally robust logistic regression.
\newblock In {\em Advances in Neural Information Processing Systems}, pages
  1576--1584, 2015.

\bibitem{agarwal2017lower}
N.~Agarwal and E.~Hazan.
\newblock Lower bounds for higher-order convex optimization.
\newblock {\em arXiv preprint arXiv:1710.10329}, 2017.

\bibitem{arjevani2019oracle}
Y.~Arjevani, O.~Shamir, and R.~Shiff.
\newblock Oracle complexity of second-order methods for smooth convex
  optimization.
\newblock {\em Mathematical Programming}, 178(1-2):327--360, 2019.

\bibitem{arjovsky2017wasserstein}
M.~Arjovsky, S.~Chintala, and L.~Bottou.
\newblock Wasserstein {G}an.
\newblock {\em arXiv preprint arXiv:1701.07875}, 2017.

\bibitem{Bilear-LinearCVG}
W.~Azizian, D.~Scieur, I.~Mitliagkas, S.~Lacoste-Julien, and G.~Gidel.
\newblock Accelerating smooth games by manipulating spectral shapes.
\newblock {\em arXiv preprint arXiv:2001.00602}, 2020.

\bibitem{bertsekas1997nonlinear}
D.P. Bertsekas.
\newblock {\em Nonlinear {P}rogramming}.
\newblock Athena Scientific, 1997.

\bibitem{carmon2017lower-1}
Y.~Carmon, J.C. Duchi, O.~Hinder, and A.~Sidford.
\newblock Lower bounds for finding stationary points {I}.
\newblock {\em Mathematical Programming}, pages 1--50, 2017.

\bibitem{carmon2017lower-2}
Y.~Carmon, J.C. Duchi, O.~Hinder, and A.~Sidford.
\newblock Lower bounds for finding stationary points {II}: {F}irst-order
  methods.
\newblock {\em Mathematical Programming}, 2019.

\bibitem{chambolle2011first}
A.~Chambolle and T.~Pock.
\newblock A first-order primal-dual algorithm for convex problems with
  applications to imaging.
\newblock {\em Journal of Mathematical Imaging and Vision}, 40(1):120--145,
  2011.

\bibitem{chambolle2016ergodic}
A.~Chambolle and T.~Pock.
\newblock On the ergodic convergence rates of a first-order primal--dual
  algorithm.
\newblock {\em Mathematical Programming}, 159(1-2):253--287, 2016.

\bibitem{gao2017first}
X.~Gao and S.~Zhang.
\newblock First-order algorithms for convex optimization with nonseparable
  objective and coupled constraints.
\newblock {\em Journal of the Operations Research Society of China},
  5(2):131--159, 2017.

\bibitem{goodfellow2014generative}
I.~Goodfellow, J.~Pouget-Abadie, M.~Mirza, B.~Xu, D.~Warde-Farley, S.~Ozair,
  A.~Courville, and Y.~Bengio.
\newblock Generative {A}dversarial {N}ets.
\newblock In {\em Advances in Neural Information Processing Systems}, pages
  2672--2680, 2014.

\bibitem{LowB-Quadratic}
A.~Ibrahim, W.~Azizian, G.~Gidel, and I.~Mitliagkas.
\newblock Linear lower bounds and conditioning of differentiable games.
\newblock {\em arXiv preprint arXiv:1906.07300}, 2019.

\bibitem{jin2019minmax}
C.~Jin, P.~Netrapalli, and M.I. Jordan.
\newblock Minmax optimization: Stable limit points of gradient descent ascent
  are locally optimal.
\newblock {\em arXiv preprint arXiv:1902.00618}, 2019.

\bibitem{jin2019local}
C.~Jin, P.~Netrapalli, and M.I. Jordan.
\newblock What is local optimality in nonconvex-nonconcave minimax
  optimization?
\newblock {\em arXiv preprint arXiv:1902.00618}, 2019.

\bibitem{juditsky2011solving}
A.~Juditsky, A.~Nemirovski, and C.~Tauvel.
\newblock Solving variational inequalities with stochastic mirror-prox
  algorithm.
\newblock {\em Stochastic Systems}, 1(1):17--58, 2011.

\bibitem{korpelevich1976extragradient}
G.M. Korpelevich.
\newblock The extragradient method for finding saddle points and other
  problems.
\newblock {\em Matecon}, 12:747--756, 1976.

\bibitem{lin2018solving}
Q.~Lin, M.~Liu, H.~Rafique, and T.~Yang.
\newblock Solving weakly-convex-weakly-concave saddle-point problems as
  weakly-monotone variational inequality.
\newblock {\em arXiv preprint arXiv:1810.10207}, 2018.

\bibitem{lin2020near}
T.~Lin, C.~Jin, and M.~Jordan.
\newblock Near-optimal algorithms for minimax optimization.
\newblock In {\em Annual Conference on Learning Theory}, 2020.

\bibitem{lin2019gradient}
T.~Lin, C.~Jin, and M.I. Jordan.
\newblock On gradient descent ascent for nonconvex-concave minimax problems.
\newblock {\em arXiv preprint arXiv:1906.00331}, 2019.

\bibitem{lu2019hybrid}
S.~Lu, I.~Tsaknakis, M.~Hong, and Y.~Chen.
\newblock Hybrid block successive approximation for one-sided non-convex
  min-max problems: algorithms and applications.
\newblock {\em arXiv preprint arXiv:1902.08294}, 2019.

\bibitem{marcotte1987note}
P.~Marcotte and J.-P. Dussault.
\newblock A note on a globally convergent newton method for solving monotone
  variational inequalities.
\newblock {\em Operations Research Letters}, 6(1):35--42, 1987.

\bibitem{mokhtari2019unified}
A.~Mokhtari, A.~Ozdaglar, and S.~Pattathil.
\newblock A unified analysis of extra-gradient and optimistic gradient methods
  for saddle point problems: Proximal point approach.
\newblock {\em arXiv preprint arXiv:1901.08511}, 2019.

\bibitem{nemirovski2004prox}
A.~Nemirovski.
\newblock Prox-method with rate of convergence $o(1/t)$ for variational
  inequalities with lipschitz continuous monotone operators and smooth
  convex-concave saddle point problems.
\newblock {\em SIAM Journal on Optimization}, 15(1):229--251, 2004.

\bibitem{nemirovsky1992information}
A.~Nemirovsky.
\newblock Information-based complexity of linear operator equations.
\newblock {\em Journal of Complexity}, 8(2):153--175, 1992.

\bibitem{nemirovsky1983problem}
A.~Nemirovsky and D.B. Yudin.
\newblock Problem complexity and method efficiency in optimization.
\newblock 1983.

\bibitem{nesterov2018implementable}
Yu. Nesterov.
\newblock Implementable tensor methods in unconstrained convex optimization.
\newblock {\em CORE Discussion Paper, 2018/05}.

\bibitem{nesterov2007dual}
Yu. Nesterov.
\newblock Dual extrapolation and its applications to solving variational
  inequalities and related problems.
\newblock {\em Mathematical Programming}, 109(2-3):319--344, 2007.

\bibitem{nesterov2018lectures}
Yu. Nesterov.
\newblock {\em Lectures on {C}onvex {O}ptimization}, volume 137.
\newblock Springer, 2018.

\bibitem{nesterov2006solving}
Yu. Nesterov and L.~Scrimali.
\newblock Solving strongly monotone variational and quasi-variational
  inequalities.
\newblock {\em Available at SSRN 970903}, 2006.

\bibitem{nisan2007algorithmic}
N.~Nisan, T.~Roughgarden, E.~Tardos, and V.~Vazirani.
\newblock {\em Algorithmic {G}ame {T}heory}.
\newblock Cambridge University Press, 2007.

\bibitem{ouyang2015accelerated}
Y.~Ouyang, Y.~Chen, G.~Lan, and E.~Pasiliao~Jr.
\newblock An accelerated linearized alternating direction method of
  multipliers.
\newblock {\em SIAM Journal on Imaging Sciences}, 8(1):644--681, 2015.

\bibitem{YYXu}
Y.~Ouyang and Y.~Xu.
\newblock Lower complexity bounds of first-order methods for convex-concave
  bilinear saddle-point problems.
\newblock {\em arXiv preprint arXiv:1808.02901}, 2018.

\bibitem{rockafellar1970convex}
R.T. Rockafellar.
\newblock {\em Convex {A}nalysis}.
\newblock Princeton University Press, 1970.

\bibitem{sanjabi2018solving}
M.~Sanjabi, M.~Razaviyayn, and J.D. Lee.
\newblock Solving non-convex non-concave min-max games under
  {P}olyak-{L}ojasiewicz condition.
\newblock {\em arXiv preprint arXiv:1812.02878}, 2018.

\bibitem{taji1993globally}
K.~Taji, M.~Fukushima, and T.~Ibaraki.
\newblock A globally convergent newton method for solving strongly monotone
  variational inequalities.
\newblock {\em Mathematical Programming}, 58(1-3):369--383, 1993.

\bibitem{von2007theory}
J.~von Neumann, O.~Morgenstern, and H.W. Kuhn.
\newblock {\em Theory of {G}ames and {E}conomic {B}ehavior (commemorative
  edition)}.
\newblock Princeton University Press, 2007.

\bibitem{wang2020improved}
Y.~Wang and J.~Li.
\newblock Improved algorithms for convex-concave minimax optimization.
\newblock {\em arXiv preprint arXiv:2006.06359}, 2020.

\bibitem{xiao2019dscovr}
L.~Xiao, A.~Yu, Q.~Lin, and W.~Chen.
\newblock {DSCOVR}: {R}andomized {P}rimal-{D}ual {B}lock {C}oordinate
  {A}lgorithms for {A}synchronous {D}istributed {O}ptimization.
\newblock {\em Journal of Machine Learning Research}, 20(43):1--58, 2019.

\bibitem{xu2017accelerated}
Y.~Xu.
\newblock Accelerated first-order primal-dual proximal methods for linearly
  constrained composite convex programming.
\newblock {\em SIAM Journal on Optimization}, 27(3):1459--1484, 2017.

\end{thebibliography}
\bibliographystyle{plain}
\appendix
\section{Proof of Lemma \ref{lemma:lowerbound-distance-prox}}
\label{appendix-lowerbound-distance-prox}
By the subspace characterization \eqref{eqn:subspace-prox}, we have
$$\|y^{2k}-\hat y^*\| \geq \sqrt{\sum^n_{j = k+1} (\hat y^*_j)^2} = \frac{q^{k}}{1-q}\sqrt{q^2 + q^4 + \cdots + q^{2(n-k)}}\geq \frac{q^k}{\sqrt{2}}\|\hat y^*\|= \frac{q^k}{\sqrt{2}}\|y^0-\hat y^*\|,$$
where the last inequality is due to the fact that $q\leq 1, k\leq \frac{n}{2}$ and $y^0 = 0$. Note that by Lemma \ref{lemma:apxsolu-proximal}, if we require $n\geq 2\log_q\left(\frac{\alpha}{4\sqrt{2}}\right)$, then we can guarantee that \begin{equation}
\label{eqn:err-n-proximal}
\|\hat y^* - y^*\| \leq \frac{q^{n+1}}{\alpha(1-q)}\leq  \frac{q^{\frac{n}{2}}}{\alpha}\cdot q^k\cdot\frac{q}{(1-q)} \leq \frac{1}{4}\cdot \frac{q^k}{\sqrt{2}}\|y^0-\hat y^*\|\quad\mbox{ for }\quad \forall 1\leq k\leq n/2,
\end{equation}
where the last inequality is due to $\frac{q^{\frac{n}{2}}}{\alpha}\leq\frac{1}{4\sqrt{2}}$ and $q/(1-q)\leq \|y^0-\hat y^*\|$. Therefore, we have
\begin{eqnarray}
\label{eqn:dist-lower-proximal}
\|y^{2k}-y^*\|^2 & \geq & (\|y^{2k}-\hat y^*\|-\|\hat y^*- y^*\|)^2\\
& \geq & \|y^{2k}-\hat y^*\|^2 - 2\|y^{2k}-\hat y^*\|\|\hat y^*- y^*\|\nonumber\\
& \geq & \min_t\left\{t^2 - 2\|\hat y^*- y^*\|t: t\geq \delta_k:= \frac{q^k}{\sqrt{2}}\|y^0-\hat y^*\|\right\}\nonumber\\
& = & \delta_k(\delta_k - 2\|\hat y^*- y^*\|)\nonumber\\
& \geq & \frac{1}{2}\delta_k^2 = \frac{q^{2k}}{4}\|y^0-\hat y^*\|^2, \nonumber
\end{eqnarray}
where the fourth line is due to that $d(t^2 - 2\|\hat y^*-y^*\|t)/dt = 2(t-\|\hat y^*-y^*\|)\geq0$ when $t\geq\delta_k$. Hence the quadratic function is monotonically increasing in the considered interval. In addition, we also have
$$\|y^0-y^*\| \leq \|y^0-\hat y^*\| + \|\hat y^*-y^*\|\leq \|y^0-\hat y^*\| + \frac{q^n}{\alpha}\cdot\frac{q}{1-q}\leq (1 + q^n/\alpha)\|y^0-\hat y^*\|\leq 2\|y^0-\hat y^*\|,$$
where the third inequality is due to that $\|y^0-\hat y^*\|\geq \hat y^*_1 = q/(1-q)$. For the last inequality, if $\alpha\geq 1$, then $q^n/\alpha<1$; if $\alpha\leq 1$, then  $q^n/\alpha\leq \alpha/32\leq1$ since $n\geq 2\log_q\left(\frac{\alpha}{4\sqrt{2}}\right)$. Combining the above two inequalities,  the desired bound \eqref{eqn:lowB-dis-prox} follows.
\section{Proof of Proposition \ref{remark:complexity-proximal}}
\label{appendix:r-complexity-proximal}
Here we only prove the last inequality of \eqref{eqn:lowerbound-proximal-1stAlgo}. Due to the fact that $(\ln(1+z))^{-1}\geq 1/z$ for $\forall z>0$, we know
\begin{eqnarray*}
(\ln(q^{-1}))^{-1} & = & (\ln(1 + (1-q)/q))^{-1} \geq \frac{q}{1-q}\\
& = &   \frac{1+\frac{2\mu_x\mu_y}{L_{xy}^2} - 2\sqrt{\left(\frac{\mu_x\mu_y}{L_{xy}^2}\right)^2 + \frac{\mu_x\mu_y}{L_{xy}^2}}}{2\sqrt{\left(\frac{\mu_x\mu_y}{L_{xy}^2}\right)^2 + \frac{\mu_x\mu_y}{L_{xy}^2}}-\frac{2\mu_x\mu_y}{L_{xy}^2}}\\
& = & \frac{\sqrt{\left(\frac{\mu_x\mu_y}{L_{xy}^2}\right)^2 + \frac{\mu_x\mu_y}{L_{xy}^2}}-\frac{\mu_x\mu_y}{L_{xy}^2}}{\frac{2\mu_x\mu_y}{L_{xy}^2}}\\
& = &  \half\sqrt{\frac{L_{xy}^2}{\mu_x\mu_y} + 1} - \frac{1}{2},\\
& = & \Omega\left(\sqrt{\frac{L_{xy}^2}{\mu_x\mu_y}}\right)
\end{eqnarray*}
which completes the proof.

\section{Proof of Theorem \ref{theorem:gen-prox}}
\label{appdx:gen-prox} 
	Before proceeding the proof, let us first quote a lemma from \cite{YYXu}. 
	\begin{lemma}[Lemma 3.1, \cite{YYXu}] 
		\label{lemma:yyx} 
		Let $\cX\subsetneqq\bar{\cX}\subseteqq\R^{p}$ be two linear subspaces. Then for any $\bar x\in\R^p$, there exists an orthogonal matrix $\Gamma\in\R^{p\times p}$ s.t. $\Gamma x = x, \forall x\in\cX$ and $\Gamma \bar x\in\bar{\cX}$.
	\end{lemma}
	Note that for an orthogonal matrix $\Gamma$, if $\Gamma x = x$, then we also have $\Gamma^\top x = x$.  Now let us start our proof of Theorem \ref{theorem:gen-prox}.

\begin{proof}
		To prove this theorem, we only need to show
		$$\{(x^0,y^0),...,(x^k,y^k)\}\subseteq U^\top\cH_x^{4k-1}\times V^\top\cH_y^{4k-1}\qquad\mbox{and}\qquad(\tilde x^k,\tilde y^k)\in U^\top\cH_x^{4k+1}\times V^\top\cH_y^{4k+1}.$$
		We separate the proof into two parts.
		
		\textbf{Part I. There exist orthogonal matrices $\hat U$, $\hat V$ s.t. when $\cA$ is applied to the rotated instance $F_{\hat U,\hat V}$, $\{(x^0,y^0),...,(x^k,y^k)\}\subseteq \hat U^\top\cH_x^{4k-1}\times \hat V^\top\cH_y^{4k-1}.$}
	
		Let $\theta = (L_{xy},\mu_x,\mu_y)$ be the set of algorithmic parameters. To prove the result, let us construct the worst-case function $F_{U,V}$ in a recursive way. 	 \\
		\textbf{Case $k = 1$:} Let us define $U_0 = V_0 = I$. When $\cA$ is applied to the function $F_{U_0,V_0}\in\cB(L_{xy}, \mu_x, \mu_y)$, the iterate sequence is 
		$(x_{0}^0, y_{0}^0) = (0, 0)$ and 
		$$\begin{cases}
		u^1_0 = \cA_u^1(\theta; x_0^0, U_0^\top A V_0 y_0^0), \qquad (x^1_0, \tilde x^1_0) = \cA_x^1(\theta; x^0_0,  U_0^\top A V_0 y_0^0, \mathbf{prox}_{\gamma_1f}(u^1_0)),\\
		v^1_0 = \cA_v^1(\theta; y_0^0, V^\top_0 A U_0x_0^0),\qquad
		(y^1_0, \tilde y^1_0) = \cA_y^1(\theta; y^0_0, V^\top_0 A U_0x_0^0, \mathbf{prox}_{\sigma_1g}(v^1_0)).
		\end{cases}$$
		By Lemma \ref{lemma:yyx}, there exists orthogonal matrices $\Gamma_x^0$ and $\Gamma_y^0$ such that 
		$\Gamma_x^0x^1_0\in\cH_x^3 = \mathrm{Span}\{Ab\}$, $\Gamma_y^0y^1_0\in\cH_y^3=\mathrm{Span}\{b, A^2b\}$, and $\Gamma_y^0b = (\Gamma_y^0)^\top b = b.$ That is 
		\begin{equation}
		\label{thm:gen-prox-1}
		x_0^1\in U_1^\top\cH_x^3,\qquad\mbox{and}\qquad y_0^1\in V_1^\top\cH_y^3, \quad V_1 b = V_1^\top b = b,
		\end{equation}	
		where $U_1 = U_0\Gamma_x^0$ and $V_1 = V_0\Gamma_y^0$. 
	
		Now we prove that when we apply the algorithm $\cA$ to $F_{U_1,V_1}$, the generated iterates $\{(x^0_1, y^0_1), (x^1_1, y^1_1)\}$ satisfy that
		$(x^0_1, y^0_1) = (0, 0)$ and  $(x^1_1, y^1_1) = (x^1_0, y^1_0)$. That is, the first two iterates generated by $\cA$ is completely the same for $F_{U_0,V_0}$ and $F_{U_1,V_1}$. The reason is because 
		$u^1_1 = \cA_u^1(\theta; x^0_1, U_1^\top AV_1y^0_1) = \cA_u^1(\theta; 0, 0) = \cA_u^1(\theta; x^0_0, U_0^\top AV_0y^0_0)  = u^1_0$, 
		therefore 
		\begin{eqnarray*}
			(x^1_1, \tilde x^1_1) & = & \cA_x^1(\theta; x^0_1, U_1^\top A V_1 y_1^0, \mathbf{prox}_{\gamma_1f}(u^1_1)) \\
			& = & \cA_x^1(\theta; 0, 0, \mathbf{prox}_{\gamma_1f}(u^1_1)) \\
			& = & \cA_x^1(\theta; x^0_0, U_0^\top A V_0 y_0^0, \mathbf{prox}_{\gamma_1f}(u^1_0)) \\
			& = & (x^1_0, \tilde x^1_0).
		\end{eqnarray*}
		Through similar argument, we know $(y^1_1,\tilde y^1_1) = (y^1_0,\tilde y^1_0)$. Therefore, \eqref{thm:gen-prox-1} indicates that 
		\begin{equation}
		\label{thm:gen-prox-2}
		x_1^1\in U_1^\top\cH_x^3,\qquad\mbox{and}\qquad y_1^1\in V_1^\top\cH_y^3, \quad V_1 b = V_1^\top b = b\in V_1^\top\cH_y^3.
		\end{equation}
		\textbf{Case $k=2$.} For the ease of the readers to follow, we perform one extra step of discussion for $k=2$, before presenting the construction on general $k$. 
	    
		For the problem instance $F_{U_1,V_1}$, the iterates generated by $\cA$ are $(x_{1}^0, y_{1}^0) = (0, 0)$ and 
		$$\begin{cases}
		u^1_1 = \cA_u^1(\theta; x_1^0, U_1^\top A V_1 y_1^0), \qquad (x^1_1, \tilde x^1_1) = \cA_x^1(\theta; x^0_1,  U_1^\top A V_1 y_1^0, \mathbf{prox}_{\gamma_1f}(u^1_1)),\\
		v^1_1 = \cA_v^1(\theta; y_1^0, V^\top_1 A U_1x_1^0),\qquad
		(y^1_1, \tilde y^1_1) = \cA_y^1(\theta; y^0_1, V^\top_1 A U_1x_1^0, \mathbf{prox}_{\sigma_1g}(v^1_1)).\\
		\end{cases}$$
		$$\begin{cases}
		u^2_1 = \cA_u^2(\theta; x_1^0, U_1^\top A V_1 y_1^0, x_1^1, U_1^\top A V_1 y_1^1), \quad (x^2_1, \tilde x^2_1) = \cA_x^2(\theta; x^0_1,  U_1^\top A V_1 y_1^0, x_1^1, U_1^\top A V_1 y_1^1, \mathbf{prox}_{\gamma_2f}(u^2_1)),\\
		v^2_1 = \cA_v^2(\theta; y_1^0, V^\top_1 A U_1x_1^0, y_1^1, V^\top_1 A U_1x_1^1),\quad
		(y^2_1, \tilde y^2_1) = \cA_y^2(\theta; y^0_1, V^\top_1 A U_1x_1^0, y_1^1, V^\top_1 A U_1x_1^1, \mathbf{prox}_{\sigma_2g}(v^2_1)).\\
		\end{cases}$$
		Note that  $x_1^1 \in U_1^\top\cH_x^3 \subsetneqq U_1^\top\cH_x^5\subsetneqq U_1^\top\cH_x^7$ and $\{y_1^1,b\}\subsetneqq V_1^\top\cH_y^3\subsetneqq V_1^\top\cH_y^5\subsetneqq V_1^\top\cH_y^7$. Therefore, there exist orthogonal matrices $\Gamma_x^1$ and $\Gamma_y^1$ such that 
		\begin{equation}
		\label{thm:gen-prox-3}
		\begin{cases}
		\Gamma_x^1x = (\Gamma_x^1)^\top x = x, \,\,\forall x\in U_1^\top\cH_x^5,\,\,\,\, \Gamma_x^1 x_1^2\in U_1^\top\cH_x^7,\\
		\Gamma_y^1y \,= (\Gamma_y^1)^\top y = y, \,\,\,\forall y\in V_1^\top\cH_y^5,\,\,\,\, \Gamma_y^1 y_1^2\in V_1^\top\cH_y^7.
		\end{cases}
		\end{equation}
		Now, let us define 
		$$U_2 = U_1\Gamma_x^1\qquad\mbox{and}\qquad V_2 = V_1\Gamma_y^1.$$
		Now we prove that if $\cA$ is applied to $F_{U_2,V_2}$, the generated iterates $\{(x_{2}^0, y_{2}^0), (x_{2}^1, y_{2}^1), (x_{2}^2, y_{2}^2)\}$ satisfy $(x_{2}^0, y_{2}^0) = (0, 0)$, $(x_{2}^1, y_{2}^1) = (x_{1}^1, y_{1}^1)$, and  $(x_{2}^2, y_{2}^2) = (x_{1}^2, y_{1}^2)$. The argument for $(x_{2}^1, y_{2}^1) = (x_{1}^1, y_{1}^1)$ is almost the same as  that of the case $k=1$. We only provide the proof for $(x_{2}^2, y_{2}^2) = (x_{1}^2, y_{1}^2)$.
	
	Next, we need to show $u_2^2 = u_1^2$, which can be proved by arguing that all the inputs to $\cA_u^2$ are the same for both $u_2^2$ and $u_1^2$. First, it is straightforward that $x_1^0 = 0 = x^0_2, U_1^\top AV_1 y_1^0 = 0 = U_2^\top AV_2 y_2^0$. By previous argument $x_2^1 = x_1^1$.  Finally, consider the last input $U_2^\top AV_2 y_2^1$, because $y_2^1 = y_1^1\in V_1^\top\cH_y^3\subsetneqq V_1^\top\cH_y^5$, we have $\Gamma_y^1 y_2^1 = y_2^1 = y_1^1\in V_1^\top\cH_y^3.$ Then $V_2y_2^1 = V_1\Gamma_y^1y_2^1\in V_1V_1^\top\cH_y^3 = \cH_y^3.$ Therefore $U_1^\top AV_2y_2^1\in U_1^\top A\cH_y^3 = U_1^\top\cH_x^5$ and 
		$$U_2^\top AV_2y_2^1 = \Gamma_x^1U_1^\top AV_2y_2^1 = U_1^\top AV_2y_2^1 = U_1^\top AV_1\Gamma_y^1y_2^1 = U_1^\top AV_1y_1^1.$$ 
		Consequently, 
		\begin{eqnarray*}
			u_2^2  = \cA_u^2(\theta; x_2^0, U_2^\top AV_2 y_2^0, x_2^1, U_2^\top AV_2 y_2^1)
			=\cA_u^2(\theta; x_1^0, U_1^\top AV_1 y_1^0, x_1^1, U_1^\top AV_1 y_1^1)
			= u_2^1
		\end{eqnarray*}
		and 
		\begin{eqnarray*}
			(x_2^2,\tilde x_2^2)  & = & \cA_x^2(\theta; x_2^0, U_2^\top AV_2 y_2^0, x_2^1, U_2^\top AV_2 y_2^1,\mathbf{prox}_{\gamma_2f}(u^2_2))\\
			& = &\cA_x^2(\theta; x_1^0, U_1^\top AV_1 y_1^0, x_1^1, U_1^\top AV_1 y_1^1,\mathbf{prox}_{\gamma_2f}(u^2_1))\\
			& = & (x_1^2,\tilde x_1^2).
		\end{eqnarray*}
		Through a similar argument, we have $(y_2^2,\tilde y_2^2) = (y_1^2,\tilde y_1^2)$. By \eqref{thm:gen-prox-2} and \eqref{thm:gen-prox-3}, we have 
		\begin{equation}
		\label{thm:gen-prox-4}
		\{x_2^0,x_2^1,x_2^2\}\in U_2^\top\cH_x^7\qquad\mbox{and}\qquad\{b,y_2^0,y_2^1,y_2^2\}\in V_2^\top\cH_y^7.
		\end{equation}
		\textbf{Case $k$.} Suppose we already have orthogonal matrices $U_{k-1}, V_{k-1}$, such that when $\cA$ is applied to $F_{U_{k-1},V_{k-1}}$, we have
		\begin{equation}
		\label{thm:gen-prox-5}
		\{x_{k-1}^0,x_{k-1}^1,\cdots, x_{k-1}^{k-1}\}\in U_{k-1}^\top\cH_x^{4k-5}\qquad\mbox{and}\qquad\{b,y_{k-1}^0,y_{k-1}^1,\cdots,y_{k-1}^{k-1}\}\in V_{k-1}^\top\cH_y^{4k-5}.
		\end{equation}
		Again, by Lemma \ref{lemma:yyx}, there exist orthogonal matrices $\Gamma_x^{k-1}$ and $\Gamma_y^{k-1}$, such that 
		\begin{equation}
		\label{thm:gen-prox-6}
		\begin{cases}
		\Gamma_x^{k-1}x = (\Gamma_x^{k-1})^\top x = x, \,\,\forall x\in U_{k-1}^\top\cH_x^{4k-3},\,\,\,\, \Gamma_x^{k-1} x_{k-1}^{k}\in U_{k-1}^\top\cH_x^{4k-1},\\
		\Gamma_y^{k-1}y \,= (\Gamma_y^{k-1})^\top y = y, \,\,\,\forall y\in V_{k-1}^\top\cH_y^{4k-3},\,\,\,\, \Gamma_y^{k-1} y_{k-1}^k\in V_{k-1}^\top\cH_y^{4k-1}.
		\end{cases}
		\end{equation}
		Now we define that 
		$$U_k = U_{k-1}\Gamma_x^{k-1} \qquad\mbox{and}\qquad V_k = V_{k-1}\Gamma_y^{k-1}.$$
		Therefore, similar to our previous discussion, we only need to argue that when $\cA$ is applied to $F_{U_k,V_k}$, the generated iterates $\{(x_k^0,y_k^0), (x_k^1,y_k^1),\cdots,(x_k^k,y_k^k)\}$ satisfy $(x_k^i,y_k^i) = (x_{k-1}^i,y_{k-1}^i)$ for $i = 0,1,...,k$. We prove this argument by induction. First, it is straightforward that $(x_k^0,y_k^0) = (0,0) = (x_{k-1}^0,y_{k-1}^0)$. Suppose $(x_k^i,y_k^i) = (x_{k-1}^i,y_{k-1}^i)$ holds for $i = 0,1,...,j-1\leq k-1$, now we prove $(x_k^j,y_k^j) = (x_{k-1}^j,y_{k-1}^j)$, which is almost identical to the case $k=2$.
	
	For any $i\in\{0,1,...,j-1\}$, let us show $U_{k-1}^\top AV_{k-1} y_{k-1}^i = U_k^\top AV_k y_k^i$. Because $y_k^i = y_{k-1}^i\in V_{k-1}^\top\cH_y^{4k-5}\subsetneqq V_{k-1}^\top\cH_y^{4k-3}$, we have $\Gamma_y^{k-1} y_k^i = y_k^i = y_{k-1}^i\in V_{k-1}^\top\cH_y^{4k-5}.$ Then $V_ky_k^i = V_{k-1}\Gamma_y^{k-1}y_k^i\in V_{k-1}V_{k-1}^\top\cH_y^{4k-5} = \cH_y^{4k-5}.$ Therefore $U_{k-1}^\top AV_{k}y_k^i\in U_{k-1}^\top A\cH_y^{4k-5} = U_{k-1}^\top\cH_x^{4k-3}$ and 
		$$U_k^\top AV_ky_k^i = (\Gamma_x^{k-1})^\top U_{k-1}^\top AV_ky_k^i = U_{k-1}^\top AV_{k}y_k^i = U_{k-1}^\top AV_{k-1}\Gamma_y^{k-1}y_k^i = U_{k-1}^\top AV_{k-1}y_{k-1}^i,$$ 
		for $0\leq i\leq j-1$. Consequently, 
		\begin{eqnarray*}
			u_k^i  & = & \cA_u^i(\theta; x_k^0, U_k^\top AV_k y_k^0,..., x_k^{i-1}, U_k^\top AV_k y_k^{i-1})\\
			& = & \cA_u^i(\theta; x_{k-1}^0, U_{k-1}^\top AV_{k-1} y_{k-1}^0,..., x_{k-1}^{i-1}, U_{k-1}^\top AV_{k-1} y_{k-1}^{i-1})\\
			& = & u_{k-1}^i
		\end{eqnarray*}
		and 
		\begin{eqnarray*}
			(x_k^i,\tilde x_k^i)  & = & \cA_x^i(\theta; x_k^0, U_k^\top AV_k y_k^0,..., x_k^{i-1}, U_k^\top AV_k y_k^{i-1},\mathbf{prox}_{\gamma_if}(u^i_k))\\
			& = &\cA_x^2(\theta; x_{k-1}^0, U_{k-1}^\top AV_{k-1} y_{k-1}^0,..., x_{k-1}^{i-1}, U_{k-1}^\top AV_{k-1} y_{k-1}^{i-1},\mathbf{prox}_{\gamma_if}(u^i_{k-1}))\\
			& = & (x_{k-1}^i,\tilde x_{k-1}^i).
		\end{eqnarray*}
		Through a similar argument, we have $(y_k^i,\tilde y_k^i) = (y_{k-1}^i,\tilde y_{k-1}^i)$. By induction, we know $(y_k^i,\tilde y_k^i) = (y_{k-1}^i,\tilde y_{k-1}^i)$ for $i = 0,1,...,k$. Consequently, we have 
		\begin{equation}
		\label{thm:gen-prox-7}
		\{x_{k}^0,x_{k}^1,\cdots, x_{k}^{k}\}\in U_{k}^\top\cH_x^{4k-1}\qquad\mbox{and}\qquad\{b,y_{k}^0,y_{k}^1,\cdots,y_{k}^{k}\}\in V_{k}^\top\cH_y^{4k-1}.
		\end{equation}
	    By setting $\hat U = U_k$ and $\hat V = V_k$, we prove the result for Part I.\vspace{0.5cm}\\       
       	\textbf{Part II. There exist orthogonal matrices $U$, $V$ such that when $\cA$ is applied to the rotated instance $F_{U,V}$, $\{(x^0,y^0),...,(x^k,y^k)\}\subseteq U^\top\cH_x^{4k-1}\times V^\top\cH_y^{4k-1},$ and $(\tilde{x}^k,\tilde{y}^k)\in U^\top\cH_x^{4k+1}\times V^\top\cH_y^{4k+1}$}.
       
       Given the result of Part I, and let $\{(x_k^0, y_k^0),...,(x_k^k, y_k^k)\}$ and $(\tilde x_k^k, \tilde y_k^k)$ be generated by $\cA$ when applied to $F_{\hat U,\hat V} = F_{U_k,V_k}$. Therefore, by Lemma \ref{lemma:yyx}, there exist orthogonal matrices $P,Q$ such that 
       \begin{equation}
       \label{thm:gen-prox-8}
       \begin{cases}
       Px = P^\top x = x, \,\,\forall x\in U_{k}^\top\cH_x^{4k-1},\,\,\,\, P \tilde x_{k}^{k}\in U_{k}^\top\cH_x^{4k+1},\\
       Qy \,=Q^\top y = y, \,\,\forall y\in V_{k}^\top\cH_y^{4k-1},\,\,\,\, Q\tilde y_{k}^k\in V_{k}^\top\cH_y^{4k+1}.
       \end{cases}
       \end{equation}
       Define $U = U_kP$, and $V = V_kQ$. Let $\{(x^0,y^0),...,(x^k,y^k)\}$ and the output $(\tilde x^k,\tilde y^k)$ be generated by $\cA$ when applied to $F_{{U,V}}$. Then following the same line of argument of Case $k$, Part I, we have 
       $$(x^i,y^i) = (x^i_k,y^i_k), \,\,\mbox{for}\,\, i = 0,1,...,k\qquad\mbox{and}\qquad(\tilde x^k,\tilde y^k) = (\tilde x^k_k,\tilde y^k_k).$$
       Therefore, combining \eqref{thm:gen-prox-8}, we complete the proof of Part II.  
\end{proof}

\section{Proof of Lemma \ref{lemma:root-estimation}}
\label{appendix-root-estimation}
For the ease of analysis, let us perform a change of variable $r:=(1-q)^{-1}$. Then the quartic equation \eqref{eqn:4EQ} can be transformed to
\begin{equation}
f(r):= 1 + \alpha r + (\beta-\alpha)r^2-2\beta r^3 + \beta r^4 = 0
\end{equation}
Although the quartic equation does have a root formula, it is impractical to use the formula %extremely complicated and will not give much insight 
for the purpose of lower iteration complexity bound. Instead, we will provide an estimation of a large enough lower bound of $r$, which corresponds to lower bound on $q$.

First, we let $\bar r = \half +\sqrt{\frac{\alpha}{\beta} + \frac{1}{4}}$. Then $f(\bar r) = 1>0.$

Second, we let $\underline r = \half +\sqrt{\frac{\alpha}{2\beta} + \frac{1}{4}}$. Then, 
\iffalse
\begin{eqnarray}
\label{eqn:inter-1***}
f(\underline r) & = & \frac{\beta}{\alpha^2}(4\beta + 2\alpha\beta -\alpha^2)\\
& = & \frac{\beta}{\alpha^2}\left(4\frac{\mu_x\mu_y}{B_xB_y} + 2\frac{\mu_x\mu_y}{B_xB_y}\left(\frac{L_{xy}^2}{4B_xB_y} + \frac{\mu_x}{B_x} + \frac{\mu_y}{B_y}\right) - \left(\frac{L_{xy}^2}{4B_xB_y} + \frac{\mu_x}{B_x} + \frac{\mu_y}{B_y}\right)^2\right)\nonumber\\
& = & \frac{\beta}{\alpha^2}\left(4\frac{\mu_x\mu_y}{B_xB_y} + \left(\frac{\mu_x\mu_y}{B_xB_y}\right)^2 - \left(\frac{\mu_x}{B_x} + \frac{\mu_y}{B_y} + \left(\frac{L_{xy}^2}{4B_xB_y}-\frac{\mu_x\mu_y}{B_xB_y}\right)\right)^2\right)\nonumber
\end{eqnarray}
\fi
\begin{eqnarray*}
	\label{eqn:inter-1}
	f(\underline r) & = & \beta\left(-\frac{\alpha^2}{4\beta^2} + \frac{1}{\beta}\right)\\
	& = & \frac{\beta}{4}\left(-\left(\frac{L_{xy}^2}{4\mu_x\mu_y} + \frac{B_x}{\mu_x} + \frac{B_y}{\mu_y}\right)^2 + \frac{4B_xB_y}{\mu_x\mu_y}\right) \\
	& = &  \frac{\beta}{4}\left(-\left(\frac{L_{xy}^2}{4\mu_x\mu_y}\right)^2 -\frac{L_{xy}^2}{2\mu_x\mu_y}\cdot\left(\frac{B_x}{\mu_x} + \frac{B_y}{\mu_y}\right)-\left(\frac{B_x}{\mu_x} - \frac{B_y}{\mu_y}\right)^2\right) \\
	&<&0.
\end{eqnarray*}

Together with the fact that $f(\bar r) = 1>0$, by continuity we know there is a root $r$ between $\left(\underline r, \bar r\right)$, where
$$\underline r = \half +\sqrt{\frac{\alpha}{2\beta} + \frac{1}{4}} = \half +\frac{1}{2\sqrt{2}}\sqrt{\frac{L_{xy}^2}{\mu_x\mu_y} + \frac{L_x}{\mu_x} + \frac{L_y}{\mu_y}}$$ 
and 
$$\bar r = \half +\sqrt{\frac{\alpha}{\beta} + \frac{1}{4}} = \half +\frac{1}{2}\sqrt{\frac{L_{xy}^2}{\mu_x\mu_y} + \frac{L_x}{\mu_x} + \frac{L_y}{\mu_y}-1}$$
This further implies $$1-\underline r^{-1}<q<1 - \bar r^{-1},$$ which proves this lemma.

\section{Proof of Lemma \ref{lemma:apxsolu-pure}}
\label{appendix-apxsolu-pure}
First, by setting $\nabla \Phi(x^*) = 0$, we get
\begin{equation}
\label{eqn:KKT-pure-1}
(B_xA^2+\mu_xI)x^*  + \frac{L_{xy}^2}{4}A(B_yA^2 + \mu_yI)^{-1}\left(Ax^* - \frac{2b}{L_{xy}}\right) = 0.
\end{equation}
Note that matrix $A$ is invertible, with
\[
A^{-1} = \begin{pmatrix}
&    &     & 1\\
&    &  1 & 1\\
&\udots & \udots&\vdots\\
1     &1 & \cdots & 1\\
\end{pmatrix}
\]
Therefore, by the interchangability of $A(B_yA^2+\mu_yI) = (B_yA^2+\mu_yI)A$, we can take the inverse and get $(B_yA^2+\mu_yI)^{-1}A^{-1} = A^{-1}(B_yA^2+\mu_yI)^{-1}$. Left multiply by $A$ and right multiply by $A$ for both sides we get the interchangablity of
\[
A(B_yA^2+\mu_yI)^{-1} = (B_yA^2+\mu_yI)^{-1}A.
\]
Applying this on equation \eqref{eqn:KKT-pure-1} and multiplying both sides by $\frac{1}{B_xB_y}(B_yA^2 + \mu_yI)$, we can equivalently write the optimality condition as
\begin{equation}
\label{eqn:KKT-pure}
(A^4  + \alpha A^2 + \beta I)x^* = \hat b
\end{equation}
where
\begin{equation*}
\alpha = \frac{L_{xy}^2}{4B_xB_y} + \frac{\mu_x}{B_x} + \frac{\mu_y}{B_y},\qquad \beta = \frac{\mu_x\mu_y}{B_xB_y},\quad \mbox{ and }\quad \hat b = \frac{L_{xy}}{2B_xB_y}Ab.
\end{equation*}

The values of matrices $A^2$ and $A^4$ can be found in \eqref{defn:A}. For the ease of discussion, we may also write equation \eqref{eqn:KKT-pure} in an expanded form as: 
\begin{eqnarray}
\label{eqn:KKT-pure-expand}
\begin{cases}
\,\,\quad\qquad\qquad\qquad\qquad(2+\alpha + \beta)x_1^*\,\,\,\,\,\,\,\,\, - (3+\alpha)x_2^*\,\,\,\,\,\, + \,\,\,\,x_3^* \,\,\,& = \hat b_1\\
\,\,\,\,\,\,\,\,\,\,\,\,\,\,-\,(3+\alpha)x_1^* \,\,\,\,\,\,+ \,\,\,(6+2\alpha + \beta)x_2^*\,\,\,\,\,\, - (4+\alpha)x_3^*\,\,\,\,\,\, +\,\,\,\, x_4^*& =  \hat b_2\\
x^*_{k-2} - (4 + \alpha)x_{k-1}^* + \,\,\,(6+2\alpha + \beta)x_k^*\,\,\,\,\,\,- (4 + \alpha)x_{k+1}^* +\,\, y_{k+2}^*& = \hat b_k\quad\mbox{ for } 3\leq k\leq n-2\\
x^*_{n-3} - (4 + \alpha)x_{n-2}^* + \,\,(6+2\alpha + \beta)x_{n-1}^*- (4 + \alpha)x_{n}^*\qquad\quad\,\,\,\,& =  \hat b_{n-1}\\
x^*_{n-2} - (4 + \alpha)x_{n-1}^* + \,\,(5+2\alpha + \beta)x_{n}^*\qquad\qquad\qquad\qquad\quad\,& =  \hat b_{n} . 
\end{cases}
\end{eqnarray}
Because $q\in(0,1)$ is a root to the  quartic equation
$1 -(4+\alpha)q  + (6+2\alpha + \beta)q^2 - (4+\alpha)q^3  + q^4 = 0$, and our approximate solution $\hat x^*$ is constructed as $\hat x^*_i = q^i$.
By direct calculation one can check that the first $n-2$ equations are satisfied and the last 2 equations are violated with controllably residuals. 
Indeed, for the $(n-1)$-th equation the violation is of the order $q^{n+1}$, and for the $n$-th equation the violation is of the order $|-q^{n} + (4+\alpha)q^{n+1} - q^{n+2}|$. Similar to the arguments for \eqref{eqn:apxsolu-proximal}, we have
\[
\beta\|\hat x^*-x^*\|\leq \|(A^4 + \alpha A^2 + \beta I)(\hat x^* - x^*)\|\leq (7+\alpha)q^n . 
\]
That is, $\|\hat x^*-x^*\|\leq \frac{7+\alpha}{\beta}\cdot q^n$, which completes the proof.

\section{Proof of Lemma \ref{lemma:lowerbound-distance-pure}}
\label{appendix-lowerbound-distance-pure}
By the subspace characterization \eqref{eqn:subspace-pure}, we have
$$\|x^k-\hat x^*\|\geq q^k\sqrt{q^2+\cdots+q^{2(n-k)}}\geq\frac{q^k}{\sqrt{2}}\|\hat x^*-x^0\|, \quad\mbox{ for }\quad \forall 1\leq k\leq n/2.$$
When we set $k\leq \frac{n}{2}$ and $n\geq 2\log_q\left(\frac{\beta}{4\sqrt{2}(7+\alpha)}\right)+2$, by \eqref{eqn:apxsolu-pure} we also have
$$\|\hat x^*-x^*\|\leq q^n(7+\alpha)/\beta\leq \frac{q^k}{4\sqrt{2}}q \leq \frac{1}{4}\cdot\frac{q^k}{\sqrt{2}}\|\hat x^*-x^0\|.$$
Therefore, similar to \eqref{eqn:dist-lower-proximal}, we also have 
\begin{equation}
\|x^k-x^*\|^2\geq \frac{q^{2k}}{16}\|x^*-x^0\|^2
\end{equation}
which proves the lemma.

\section{Proof of $\ln(2ac^2)  = \Omega(1)$}
\label{appdx:Omega-1}
\begin{proof} 
	Note that $a = \min\{c^{-2},d^{-2}\}$, if $c^{-2}\leq d^{-2}$, then $ac^2 = 1$. Consequently, 
	$$\ln\left(2ac^2\right) = \ln 2 = \Omega(1).$$
	However, when $c^{-2}\geq d^{-2}$, the situation is more complicated. In this case, 
	$$ac^2 = \frac{c^2}{d^2} = \frac{R_y^2}{R_x^2}\cdot\frac{\|\hat x^*\|^2}{\|\hat y^*\|^2},$$
	where $\hat x^*$ and $\hat y^*$ is the solution to the unscaled worst-case instance $\hat F_\epsilon\in\cF(L_x,L_y,L_{xy},\mu_x,\mu_y)$. For the ease of discussion, let us take the dimension $n$ is sufficiently large so that we can view the approximate solution constructed in Lemma \ref{lemma:apxsolu-pure} as the exact solution. Therefore, we have 
	$$\begin{cases}
	\hat x^*(i) = {q^i},  \,\, i = 1,...,n\\
	(\mu_yI + B_yA^2)\hat y^* = \frac{L_{xy}}{2}A\hat x^* - b, 
	\end{cases}$$
    where $q$ is defined by Theorem \ref{thm:pure} and the second equality is due to the first-order stationary condition. Note that equation \eqref{eqn:KKT-pure-1} also provides that 
    $$(B_xA^2+\mu_xI)\hat x^*  + \frac{L_{xy}^2}{4}A(B_yA^2 + \mu_yI)^{-1}\left(A\hat x^* - \frac{2b}{L_{xy}}\right) = 0.$$
    Combining the above two relations, we have 
    \begin{eqnarray}
    \hat y^* & = & (\mu_yI + B_yA^2)^{-1}(\frac{L_{xy}}{2}A\hat x^* - b)\nonumber\\
    & = & -\frac{2}{L_{xy}}A^{-1}(B_xA^2+\mu_xI)\hat x^*\nonumber\\
    & = & -\frac{2B_x}{L_{xy}}A\hat x^*-\frac{128\epsilon}{L_{xy}R_x^2}A^{-1}\hat x^*.\nonumber
    \end{eqnarray}
    Substituting the specific forms of $A$ and $A^{-1}$, we have 
    $$\hat y^*(i) = \begin{cases}
    -\frac{2B_x}{L_{xy}} q^n - \frac{128\epsilon}{L_{xy}R_x^2}q^n, \quad i = 1\\
    -\frac{2B_x}{L_{xy}} q^{n+1-i}(1-q)- \frac{128\epsilon}{L_{xy}R_x^2}q^{n+1-i}\frac{1-q^i}{1-q}, \quad i\geq 2.
    \end{cases}$$
    Therefore, we have 
    \begin{eqnarray*}
    	\|\hat y^*\|^2 \leq  \left(\frac{2B_x}{L_{xy}} + \frac{128\epsilon}{L_{xy}R_x^2}\right)^2q^{2n} + \left(\frac{2B_x}{L_{xy}} (1-q)+ \frac{128\epsilon}{L_{xy}R_x^2(1-q)}\right)^2\sum^{n}_{i=1} q^{2i}.
    \end{eqnarray*}
    For ease of discussion, the following simplifications are made. First, we omit the $q^{2n}$ term since $q<1$ and $n$ is sufficiently large. Second, note that Lemma \ref{lemma:root-estimation} indicates that $1-q = \Theta(\epsilon)$, the term $\frac{2B_x}{L_{xy}} (1-q) = \cO(\epsilon)$ and the term $\frac{128\epsilon}{L_{xy}R_x^2(1-q)} = \Omega(1)$. Thus we also omit the $\frac{2B_x}{L_{xy}} (1-q)$ term which is significantly smaller. Therefore, we can write 
    \begin{eqnarray*}
    	\|\hat y^*\|^2 \leq  \left(\frac{128\epsilon}{L_{xy}R_x^2(1-q)}\right)^2\sum^{n}_{i=1} q^{2i} = \left(\frac{128\epsilon}{L_{xy}R_x^2(1-q)}\right)^2\|\hat x^*\|^2.
    \end{eqnarray*}
    As a result, 
    $$ac^2 = \frac{R_y^2}{R_x^2}\cdot\frac{\|\hat x^*\|^2}{\|\hat y^*\|^2} \geq \frac{L_{xy}^2R_y^2R_x^2(1-q)^2}{128^2 \epsilon^2}.$$
    In Lemma \ref{lemma:root-estimation}, we also have a lower bound of $1-q$ as
    $$1-q>\left(\half +\frac{1}{2}\sqrt{\frac{L_{xy}^2}{\mu_x\mu_y} + \frac{L_x}{\mu_x} + \frac{L_y}{\mu_y}-1}\right)^{-1} \overset{(i)}{>}\frac{128\epsilon}{L_{xy}R_xR_y}$$
    where (i) is because we have omitted the terms of smaller magnitude.
    Therefore,
    $$\ln\left(2ac^2\right) \geq \ln\left(\frac{2L_{xy}^2R_y^2R_x^2}{128^2 \epsilon^2}\cdot\frac{128^2\epsilon^2}{L_{xy}^2R_x^2R_y^2}\right) =  \ln\left(2\right) = \Omega(1).$$
    Thus we complete the proof.  
\end{proof}

\end{document}